\documentclass[a4paper,11pt]{article}

\pdfoutput=1 
\usepackage[utf8]{inputenc}
\usepackage[british]{babel}
\usepackage{amsthm}
\usepackage{amsmath}
\usepackage{amssymb}
\usepackage{tikz}
\usepackage{booktabs}
\usepackage{array}
\usepackage[font={small},width=0.9\textwidth]{caption}
\usepackage{tabularray}

\newtheorem{theorem}{Theorem}[section]

\newtheorem{lemma}[theorem]{Lemma}

\theoremstyle{definition}
\newtheorem{example}[theorem]{Example}

\begin{document}

\begin{center}
{\Large\textbf{The Pattern Complexity of the Squiral Tiling}}

\vspace{2ex}
Johan Nilsson
\end{center}

\begin{abstract}
	We give an exact formula for the number of distinct square patterns of a given size that occur in the Squiral tiling. 
\end{abstract}


\section{Introduction}
The squiral tiling can be defined as a block substitution on the binary alphabet $\mathcal{A} = \{\mathtt{0,1}\}$ via
\begin{equation}
	\label{def:mu}
	\mu :  \quad 
	\begin{tikzpicture}[scale=0.45,baseline={([yshift=-.8ex]current bounding box.center)}]]
		\filldraw[draw=black,fill = yellow] (0,0) rectangle (1,1);
		\draw (0.5,0.5) node{\small\texttt{0}};
	\end{tikzpicture}
	\quad
	\mapsto
	\quad
	\begin{tikzpicture}[scale=0.45,baseline={([yshift=-.8ex]current bounding box.center)}]]
		\fill[yellow] (0,0) rectangle (3,3);
		\fill[red] (0,0) rectangle ++(1,1);
		\fill[red] (2,0) rectangle ++(1,1);
		\fill[red] (0,2) rectangle ++(1,1);
		\fill[red] (2,2) rectangle ++(1,1);
		\draw[cap = rect] (0,0) grid (3,3);
		\draw (0.5,0.5) node{\small\texttt{1}};
		\draw (1.5,0.5) node{\small\texttt{0}};
		\draw (2.5,0.5) node{\small\texttt{1}};
		\draw (0.5,1.5) node{\small\texttt{0}};
		\draw (1.5,1.5) node{\small\texttt{0}};
		\draw (2.5,1.5) node{\small\texttt{0}};
		\draw (0.5,2.5) node{\small\texttt{1}};
		\draw (1.5,2.5) node{\small\texttt{0}};
		\draw (2.5,2.5) node{\small\texttt{1}};
	\end{tikzpicture}
	\ ,
	\quad
	\quad
	\begin{tikzpicture}[scale=0.45,baseline={([yshift=-.8ex]current bounding box.center)}]]
		\filldraw[draw=black,fill = red] (0,0) rectangle (1,1);
		\draw (0.5,0.5) node{\small\texttt{1}};
	\end{tikzpicture}
	\quad
	\mapsto
	\quad
	\begin{tikzpicture}[scale=0.45,baseline={([yshift=-.8ex]current bounding box.center)}]]
		\fill[red] (0,0) rectangle (3,3);
		\fill[yellow] (0,0) rectangle ++(1,1);
		\fill[yellow] (2,0) rectangle ++(1,1);
		\fill[yellow] (0,2) rectangle ++(1,1);
		\fill[yellow] (2,2) rectangle ++(1,1);
		\draw[cap = rect] (0,0) grid (3,3);
		\draw (0.5,0.5) node{\small\texttt{0}};
		\draw (1.5,0.5) node{\small\texttt{1}};
		\draw (2.5,0.5) node{\small\texttt{0}};
		\draw (0.5,1.5) node{\small\texttt{1}};
		\draw (1.5,1.5) node{\small\texttt{1}};
		\draw (2.5,1.5) node{\small\texttt{1}};
		\draw (0.5,2.5) node{\small\texttt{0}};
		\draw (1.5,2.5) node{\small\texttt{1}};
		\draw (2.5,2.5) node{\small\texttt{0}};
	\end{tikzpicture}
	\ ,
\end{equation}
see \cite{BG13, FHG, GS87} and further references therein. Let us by $T$ denote the limit pattern obtain, when taking the letter $\mathtt{0}$ as starting seed and apply $\mu$ repeatedly. See Figure~\ref{fig:firstTn} for some of the first iteration of $\mu$ on the seed $\mathtt{0}$. We refer to $T$ as the squiral tiling. 

In this paper we focus on the pattern complexity of the squiral tiling, that is, we look at the number of distinct square patterns of a given size that occur anywhere in $T$. The main result of this paper is the following theorem.

\begin{theorem}
	\label{thm:main}
	Let $A_n$ be the number of unique patterns of size $n \times n$ that occur in the sqiural tiling. Then $A_1=2$, $A_2 = 14$, $A_3 = 70$, and 
	\begin{equation}
		\label{eq:main}
		A_n = \big(4+8\alpha -8\beta\big)(n-1)^2 + \left(12\cdot3^{\alpha}+24\cdot3^{\beta}\right)(n-1) -18 \cdot 9^{\alpha},
	\end{equation}
	for $n \geq 4$, where $\alpha = \lfloor \log_3 (n-2)\rfloor$  and $\beta = \lfloor \log_3 \frac{n-2}{2}\rfloor$.
\end{theorem}

Similar results to Theorem~\ref{thm:main} have been given by Allouche \cite{allouche}, Nilsson \cite {nilsson}, and by Galanov \cite{galanov}. The work by Galanov focuses on the pattern complexity in the Robinson tiling (see also~\cite{robinson}), while Allouche considers the number of distinct patterns occurring in the classical paperfolding sequences and their generalizations. Nilsson expands Allouche's results to the 2 dimensional case. Our work here follows a similar line of ideas as applied by Nilsson in \cite{nilsson}. 

The article is organized as follows; in the next section we introduce necessary notations, and give a few preliminary results. Thereafter, in section~\ref{sec:recursion}, we derive a system of recursions describing the size of sets of distinct patterns. The proof of our main result is then completed in section~\ref{sec:mainproof}.

\section{Preliminaries}

Recall the definition of $\mu$ on the alphabet $\mathcal{A} = \{\mathtt{0,1}\}$ from \eqref{def:mu}. An object of the form $\mu^n(x)$ where $x\in \mathcal{A}$ and $n \geq 0$ is called a \emph{supertile}. Here $\mu^n = \mu^{n-1}\circ \mu$ and $\mu^0 = Id$. Define the particular supertiles $T_n$ by $T_n := \mu^n(\mathtt{0})$, for $n\geq 0$. This definition can also be written as the block recursion 
\begin{equation}
\renewcommand{\arraystretch}{1.3}
\label{eq:DefTnBlock}
	T_{n+1} =  \begin{array}{ccc} \mu^n(\mathtt{1}) & T_n &\mu^n(\mathtt{1}) \\ T_n  & T_n  & T_n \\ \mu^n(\mathtt{1}) & T_n &\mu^n(\mathtt{1}) \end{array},
\end{equation}
for $n\geq 0$ and with $T_0 = \mathtt{0}$. See Figure~\ref{fig:firstTn} for a visualisation of the first $T_n$s. 

\begin{figure}[ht]
\centering\small
	\begin{tabular}{*{4}{c}}
		\begin{tikzpicture}[scale = 0.25, baseline={([yshift=-.8ex]current bounding box.center)}]
	\fill[yellow] (0,0) rectangle (1,1);
	\draw[cap = rect] (0,0) grid (1,1);
	\draw(0.5, -2) node{$T_0$};
\end{tikzpicture}  & 
		\begin{tikzpicture}[scale = 0.25, baseline={([yshift=-.8ex]current bounding box.center)}]
	\fill[yellow] (0,0) rectangle (3,3);
	\fill[red] (0,2) rectangle ++(1,1);
	\fill[red] (2,2) rectangle ++(1,1);
	\fill[red] (0,0) rectangle ++(1,1);
	\fill[red] (2,0) rectangle ++(1,1);
	\draw[cap = rect] (0,0) grid (3,3);
	\draw(1.5, -2) node{$T_1$};
\end{tikzpicture}  &
		\begin{tikzpicture}[scale = 0.25, baseline={([yshift=-.8ex]current bounding box.center)}]
	\fill[yellow] (0,0) rectangle (9,9);
	\fill[red] (1,8) rectangle ++(1,1);
	\fill[red] (3,8) rectangle ++(1,1);
	\fill[red] (5,8) rectangle ++(1,1);
	\fill[red] (7,8) rectangle ++(1,1);
	\fill[red] (0,7) rectangle ++(1,1);
	\fill[red] (1,7) rectangle ++(1,1);
	\fill[red] (2,7) rectangle ++(1,1);
	\fill[red] (6,7) rectangle ++(1,1);
	\fill[red] (7,7) rectangle ++(1,1);
	\fill[red] (8,7) rectangle ++(1,1);
	\fill[red] (1,6) rectangle ++(1,1);
	\fill[red] (3,6) rectangle ++(1,1);
	\fill[red] (5,6) rectangle ++(1,1);
	\fill[red] (7,6) rectangle ++(1,1);
	\fill[red] (0,5) rectangle ++(1,1);
	\fill[red] (2,5) rectangle ++(1,1);
	\fill[red] (3,5) rectangle ++(1,1);
	\fill[red] (5,5) rectangle ++(1,1);
	\fill[red] (6,5) rectangle ++(1,1);
	\fill[red] (8,5) rectangle ++(1,1);
	\fill[red] (0,3) rectangle ++(1,1);
	\fill[red] (2,3) rectangle ++(1,1);
	\fill[red] (3,3) rectangle ++(1,1);
	\fill[red] (5,3) rectangle ++(1,1);
	\fill[red] (6,3) rectangle ++(1,1);
	\fill[red] (8,3) rectangle ++(1,1);
	\fill[red] (1,2) rectangle ++(1,1);
	\fill[red] (3,2) rectangle ++(1,1);
	\fill[red] (5,2) rectangle ++(1,1);
	\fill[red] (7,2) rectangle ++(1,1);
	\fill[red] (0,1) rectangle ++(1,1);
	\fill[red] (1,1) rectangle ++(1,1);
	\fill[red] (2,1) rectangle ++(1,1);
	\fill[red] (6,1) rectangle ++(1,1);
	\fill[red] (7,1) rectangle ++(1,1);
	\fill[red] (8,1) rectangle ++(1,1);
	\fill[red] (1,0) rectangle ++(1,1);
	\fill[red] (3,0) rectangle ++(1,1);
	\fill[red] (5,0) rectangle ++(1,1);
	\fill[red] (7,0) rectangle ++(1,1);
	\draw[cap = rect] (0,0) grid (9,9);
	\draw(4.5, -2) node{$T_2$};
\end{tikzpicture}  &
		\input{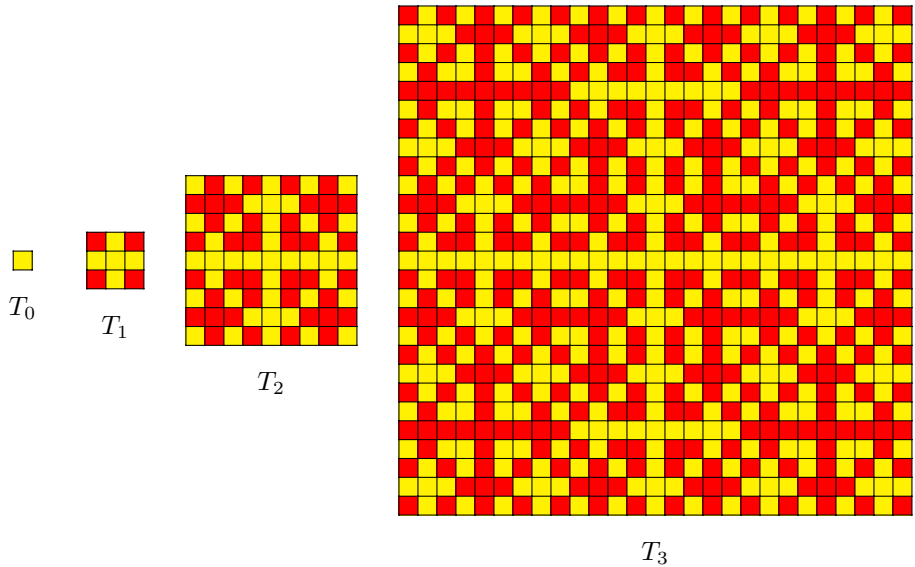}  
	\end{tabular}
	\caption{The first $T_n := \mu^{n}(\mathtt{0})$.}
	\label{fig:firstTn}
\end{figure}

By $T$ we shall mean the supertile of infinite order, obtained as the limit of the sequence $(T_n)_{n\geq0}$. We refer to $T$ as the \emph{squiral tiling}. 
Note that $T_n$ can be seen as a binary matrix, compare \eqref{def:mu}. According to the language used in the field of tilings, we say that submatrices of the $T_n$s are called \emph{patterns} or \emph{subpatterns}. (In the literature the term \emph{patch} (see \cite{FHG,BG13}) is also commonly used for this.) Clearly, any $T_n$ is also a pattern. For completeness, we also say that $T$ is a pattern (an infinite one).  We also adopt the notations used for matrices, with rows and columns. This means we can describe a finite pattern $S$ via its entries, that is, $S_{r,c} \in \mathcal{A}$ is the entry in $S$ at row $r$ and column $c$.
For a pattern $S$ (finite or infinite), we let $P(S, m\times n)$, where $m$ and $n$ are positive integers, be the set of all $m\times n$ patterns that occur somewhere in $S$. In the case of $S$ being a finite pattern, we use the notation $S[r, c, n \times k]$ to denote the $n \times k$ subpattern of $S$ that has its upper left corner at row $r$ and column $c$ in $S$. The notation $|\cdot|$ denotes the cardinality of a set. By the definition in \eqref{eq:DefTnBlock} we obtain the following result. 
\begin{lemma}
	\label{lemma:inclusion}
	Let $n\geq 0$. Then $T_n \in P(T_{n+1}, 3^n\times 3^n)$.\hfill$\square$ 
\end{lemma}

The Lemma~\ref{lemma:inclusion} shows that the chain of nested sets of subpatterns,
\[
   P(T_0,m \times m) \subseteq \cdots  \subseteq P(T_n,m \times m) \subseteq P(T_{n+1},m \times m) \subseteq \cdots, 
\]
is monotonic including in $n$, (if $m\leq3^n$). Next, we show that this chain is strictly monotonic including until all possible subpatterns are contained. 

\begin{lemma}
	\label{lemma:uniqueplateau}
	Let $m\geq1$. If there is an $n\geq0$ such that 
	\begin{equation}
		\label{eq:plateauassumption}
		P(T_{n}, m \times m) = P(T_{n+1}, m \times m),
	\end{equation}
	with $m\leq 3^n$, then 
	\begin{equation}
		\label{eq:plateauresult}
		P(T_{n}, m \times m) = P(T_{n+k}, m \times m)
	\end{equation}
	for all integers $k \geq 1$, and in particular $P(T_{n}, m\times m) = P(T, m\times m)$. 
\end{lemma}

\begin{proof}
We give a proof by induction on $k$ in \eqref{eq:plateauresult}. The basis case, $k = 1$, is direct from the assumption \eqref{eq:plateauassumption}. Assume for induction that \eqref{eq:plateauresult} holds for $1\leq k \leq p$. 
	
For the induction step, $k = p+1$, consider a pattern $a \in P(T_{n+p+1}, m \times m)$. Then there is a pattern $b \in P(T_{n+p}, m \times m)$ such that $a$ is a subpattern of $\mu(b)$. By the induction assumption we have that $b \in P(T_{n+p-1}, m \times m)$. This implies
\[  a \in P(\mu(b),m \times m) \subseteq P(T_{n+p}, m \times m).
\]
Therefore $P(T_{n+p}, m \times m) \supseteq P(T_{n+p+1}, m \times m)$, and by Lemma~\ref{lemma:inclusion} it follows that 
\[  P(T_{n+p}, m \times m) = P(T_{n+p+1}, m \times m),
\]
which completes the induction. 
\end{proof}

\begin{example}
\label{ex:patterncount}
By inspection, we find
\[
	P(T_2, 2\times 2) = P(T_3, 2\times 2),
\]
with $|P(T_2, 2\times 2)| = 14$. Lemma~\ref{lemma:uniqueplateau} now implies that $P(T_2, 2\times2) = P(T, 2\times2)$, so we can find all $2\times 2$ patterns in the squiral tiling $T$ by just looking at patterns in $T_2$. In the same way, continuing the enumeration and applying Lemma~\ref{lemma:uniqueplateau}, we find 
\[  
	P(T_3, 4\times4) = P(T_4, 4\times4) =  P(T, 4\times4),
\]
with $|P(T, 4\times4)| = 126$. As a consequence, we clearly also have $P(T_4, 3\times3) = P(T, 3\times3)$ without any further enumerations. This because $T_4$ contains all $4\times4$ patterns, and therefore it must also contain all $3\times3$ patterns. \hfill$\diamond$
\end{example}

The elements of the set $P(T,m \times n)$ can be split into sets depending on their position relative to the underlying structure of supertiles of size $3\times 3$. For $i,j\in\{1,2,3\}$ we define the sets  
\begin{equation}
\label{eq:defofPij}
	P_{i,j}(T, m \times n) 
	:= \left\{ \mu(x)[i,j,m \times n] : x \in P(T, m \times n) \right\}.
\end{equation}
The definition in \eqref{eq:defofPij} can be extend to all positive indices via
\[
	P_{i+3s,j+3t}(T, m \times n) := P_{i,j}(T, m \times n),
\]
where $s,t \in \mathbb{N}$. It is clear that 
\[
	P(T, m \times n) = \bigcup_{i,j\in\{1,2,3\}} P_{i,j}(T, m \times n),
\]
as any $x\in P(T, m \times n)$  must be in at least one $P_{i,j}(T, m \times n)$. Moreover, it is by construction clear that each of the sets $P_{i,j}(T, m \times n)$ are non-empty. 
 
\begin{example}
\label{ex:disjointPij}
In Example~\ref{ex:patterncount} we saw that all patterns of size $4\times4 $ are found in $T_3$. This leads to that we can find all the sets $P_{i,j}(T, 4 \times 4)$, for $i,j \in \{1,2,3\}$. An enumeration shows that 
\[	
	P(T, 4 \times 4) = \bigcup_{i,j\in\{1,2,3\}} P_{i,j}(T, 4 \times 4), 
\]
and that the sets on the right hand side are pairwise disjoint. Moreover, we find $|P_{i,j}(T, 4 \times 4)| = 14$, for all the indices involved. This gives $|P(T, 4\times4)| = 9 \cdot 14 = 126$, as also already seen in Example~\ref{ex:patterncount}.\hfill$\diamond$
\end{example}

\begin{lemma}
\label{lemma:disjointPij}
Let $n,m\geq4$. Then 
\begin{equation}
\label{eq:disjointPij}
	P(T, m \times n) = \bigcup_{i,j\in\{1,2,3\}} P_{i,j}(T, m \times n), 
\end{equation}
and the sets on the right hand side of \eqref{eq:disjointPij} are non-empty and pairwise disjoint.
\end{lemma}
\begin{proof}
Assume for contradiction that there are $m,n\geq 4$ and two different pairs of pair of indices $i_1,j_1,i_2,j_2\in\{1,2,3\}$ such that there is a pattern 
\[	
	x \in P_{i_1,j_1}(T, m \times n) \bigcap P_{i_2,j_2}(T, m \times n).
\]
Then the pattern $x' = x[1,1,4\times4]$ must be in the intersection 
\[
	P_{i_1,j_1}(T, 4 \times 4) \bigcap P_{i_2,j_2}(T, 4 \times 4),
\]
but according to what we saw in Example~\ref{ex:disjointPij}, this intersection is empty.
\end{proof}

\begin{example}
\label{ex:extension}
A computer enumeration shows that 
\[
	|P_{3,3}(T, 5\times 5)| = |P_{1,1}(T, 9\times 9)|.
\]
See Figure~\ref{fig:extension} for the outlay of the patterns in the sets above in relation to each other. It follows that we may take $P_{3,3}(T, 5\times 5) $ and extend it's contained patterns with 2 extra rows and columns on either side, without changing the cardinality of the set. That is, for $i,j\in\{1,2,3\}$ we have 
\[
	|P_{3,3}(T, 5\times 5)| = |P_{i,j}(T, m\times n)|
\]
where $m \in\{8-i,9-i,10-i\}$, $n\in\{8-j,9-j,10-j\}$.\hfill$\diamond$
\end{example}

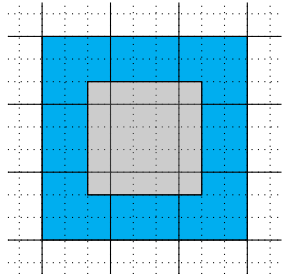
\begin{figure}[ht]
	\begin{center}
		\begin{tikzpicture}[scale = 0.3]
			\draw[color=black, fill=cyan, even odd rule ] (1, 1) rectangle ++ (9, 9) (3, 3) rectangle ++ (5, 5);			
			\draw[step=1, black, dotted]    (-0.5,-0.5) grid (11.5,11.5);
			\foreach \a in {1,4,7,10}
			{	\draw (-0.5,\a) -- (11.5,\a);
				\draw (\a, -0.5) -- (\a, 11.5);
			}
			\draw[color=black, fill=gray, fill opacity=0.4] (3, 3) rectangle ++ (5, 5);
			
		\end{tikzpicture}
		\caption{The extension of patterns, as discussed in Example~\ref{ex:extension}. The gray region represent an element in $P_{3,3}(T,5\times5)$, and the blue region one in $P_{1,1}(T,9\times9)$. An enumeration shows that the two sets have the same cardinality. The solid grid indicates the structure of supertiles of size $3\times3$.}
     \label{fig:extension}
	\end{center}
\end{figure}

The extension of patterns observed in Example~\ref{ex:extension} can be extend to more general cases, as stated in the following lemma.

\begin{lemma}
\label{lemma:extension}
Let $i,j\in\{1,2,3\}$. Then 
\[
	|P_{3,3}(T, (5+3s)\times (5+3t)| = |P_{i,j}(T, (m+3s)\times (n+3t))|
\]
where $m \in\{8-i,9-i,10-i\}$, $n\in\{8-j,9-j,10-j\}$, and $s, t\in \mathbf{N}$. \hfill$\square$
\end{lemma}

\section{Recursion}
\label{sec:recursion}

In this section we turn to the question of deriving a system of recursions describing the size of the set $P(T, n \times  n )$. Let us start by introducing the following notations, 
\begin{align*}
	A_n &:= |P(T, n \times  n )|,   \\
	B_n &:= |P(T, n \times (n+1))|, \\
	C_n &:= |P(T,(n+1)\times n )|. 
\end{align*}
Note here that the quantity $A_n$ is the one used in the formulation of Theorem~\ref{thm:main}. 

The next step is now to derive recursion relations for $A_n, B_n$, and $C_n$. Below we go through the different cases involved in the recursions. If $n\geq2$ we can apply Lemma~\ref{lemma:disjointPij} and Lemma~\ref{lemma:extension} to obtain
\begin{equation}
\label{eq:recA3n2}
\renewcommand{\arraystretch}{1.5}
\text{\small
	$\displaystyle{A_{3n-2} = \sum_{i,j\in\{1,2,3\}} |P_{i,j}(T,(3n-2)\times (3n-2))| = 9 \cdot  A_n.}$
}
\end{equation}
See Figure~\ref{fig:recA3n2} for a visualization of these extension of the $P$-sets.

\begin{figure}[ht]
	\begin{center}
		\begin{tikzpicture}[scale = 0.25]
			\begin{scope}[xshift = 0cm, yshift = 28cm]			
				\draw[color=black, fill=gray, fill opacity=0.3] (0, 5) rectangle ++ (4, 4);
				\draw[color=black, fill=cyan, even odd rule] (0, 5) rectangle ++ (4, 4) (0, 3) rectangle ++ (6, 6);
				\draw[step=1, black, dotted]    (-0.5,-0.5) grid (9.5, 9.5);
				\draw[step=3, black, cap=rect]  (-0.5,-0.5) grid (9.5, 9.5);
				\draw( 4.5, -2) node{\footnotesize $(i,j) = (1,1)$};
			\end{scope}
			\begin{scope}[xshift = 14 cm, yshift = 28cm]
				\draw[color=black, fill=gray, fill opacity=0.3] (1, 5) rectangle ++ (4, 4);
				\draw[color=black, fill=cyan, even odd rule] (1, 5) rectangle ++ (4, 4) (0, 3) rectangle ++ (6, 6);
				\draw[step=1, black, dotted]    (-0.5,-0.5) grid (9.5, 9.5);
				\draw[step=3, black, cap=rect]  (-0.5,-0.5) grid (9.5, 9.5);
				\draw( 4.5, -2) node{\footnotesize $(i,j) = (1,2)$};
			\end{scope}
			\begin{scope}[xshift = 28 cm, yshift = 28cm]
				\draw[color=black, fill=gray, fill opacity=0.3] (2, 5) rectangle ++ (4, 4);
				\draw[color=black, fill=cyan, even odd rule] (2, 5) rectangle ++ (4, 4)(0, 3) rectangle ++ (6, 6);
				\draw[step=1, black, dotted]    (-0.5,-0.5) grid (9.5, 9.5);
				\draw[step=3, black, cap=rect]  (-0.5,-0.5) grid (9.5, 9.5);
				\draw( 4.5, -2) node{\footnotesize $(i,j) = (1,3)$};
			\end{scope}
			\begin{scope}[xshift = 0cm, yshift = 14cm]			
				\draw[color=black, fill=gray, fill opacity=0.3] (0, 4) rectangle ++ (4, 4);
				\draw[color=black, fill=cyan, even odd rule] (0, 4) rectangle ++ (4, 4) (0, 3) rectangle ++ (6, 6);
				\draw[step=1, black, dotted]    (-0.5,-0.5) grid (9.5, 9.5);
				\draw[step=3, black, cap=rect]  (-0.5,-0.5) grid (9.5, 9.5);
				\draw( 4.5, -2) node{\footnotesize $(i,j) = (2,1)$};
			\end{scope}
			\begin{scope}[xshift = 14 cm, yshift = 14cm]
				\draw[color=black, fill=gray, fill opacity=0.3] (1, 4) rectangle ++ (4, 4);
				\draw[color=black, fill=cyan, even odd rule] (1, 4) rectangle ++ (4, 4) (0, 3) rectangle ++ (6, 6);
				\draw[step=1, black, dotted]    (-0.5,-0.5) grid (9.5, 9.5);
				\draw[step=3, black, cap=rect]  (-0.5,-0.5) grid (9.5, 9.5);
				\draw( 4.5, -2) node{\footnotesize $(i,j) = (2,2)$};
			\end{scope}
			\begin{scope}[xshift = 28 cm, yshift = 14cm]
				\draw[color=black, fill=gray, fill opacity=0.3] (2, 4) rectangle ++ (4, 4);
				\draw[color=black, fill=cyan, even odd rule] (2, 4) rectangle ++ (4, 4) (0, 3) rectangle ++ (6, 6);
				\draw[step=1, black, dotted]    (-0.5,-0.5) grid (9.5, 9.5);
				\draw[step=3, black, cap=rect]  (-0.5,-0.5) grid (9.5, 9.5);
				\draw( 4.5, -2) node{\footnotesize $(i,j) = (2,3)$};
			\end{scope}
			\begin{scope}[xshift = 0cm, yshift = 0cm]			
				\draw[color=black, fill=gray, fill opacity=0.3] (0, 3) rectangle ++ (4, 4);
				\draw[color=black, fill=cyan, even odd rule] (0, 3) rectangle ++ (4, 4) (0, 3) rectangle ++ (6, 6);
				\draw[step=1, black, dotted]    (-0.5,-0.5) grid (9.5, 9.5);
				\draw[step=3, black, cap=rect]  (-0.5,-0.5) grid (9.5, 9.5);
				\draw( 4.5, -2) node{\footnotesize $(i,j) = (3,1)$};
			\end{scope}
			\begin{scope}[xshift = 14 cm, yshift = 0cm]
				\draw[color=black, fill=gray, fill opacity=0.3] (1, 3) rectangle ++ (4, 4);
				\draw[color=black, fill=cyan, even odd rule] (1, 3) rectangle ++ (4, 4) (0, 3) rectangle ++ (6, 6);
				\draw[step=1, black, dotted]    (-0.5,-0.5) grid (9.5, 9.5);
				\draw[step=3, black, cap=rect]  (-0.5,-0.5) grid (9.5, 9.5);
				\draw( 4.5, -2) node{\footnotesize $(i,j) = (3,2)$};
			\end{scope}
			\begin{scope}[xshift = 28 cm, yshift = 0cm]
				\draw[color=black, fill=gray, fill opacity=0.3] (2, 3) rectangle ++ (4, 4);
				\draw[color=black, fill=cyan, even odd rule] (2, 3) rectangle ++ (4, 4) (0, 3) rectangle ++ (6, 6);
				\draw[step=1, black, dotted]    (-0.5,-0.5) grid (9.5, 9.5);
				\draw[step=3, black, cap=rect]  (-0.5,-0.5) grid (9.5, 9.5);
				\draw( 4.5, -2) node{\footnotesize $(i,j) = (3,3)$};
			\end{scope}
		\end{tikzpicture}
		\caption{Extensions used in \eqref{eq:recA3n2} for the recursion of $A_{3n-2}$. The gray region represent an element in $P_{i,j}(T,(3n-2)\times(3n-2))$, and the blue one the region it is extend with.}
     \label{fig:recA3n2}
	\end{center}
\end{figure}
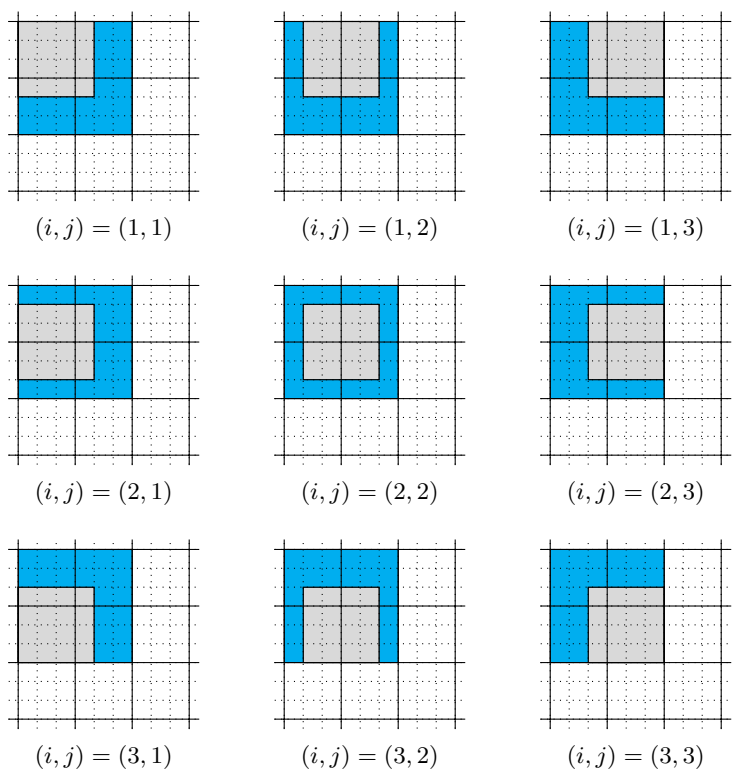

\begin{equation}
\label{eq:recA3n1}
\renewcommand{\arraystretch}{1.5}
\text{\small$\begin{array}{c@{\ }c@{\ }c@{\ }l}
	A_{3n-1} 
	&=& \multicolumn{2}{l}{ \displaystyle{\sum_{i=1}^3 \sum_{j=1}^3} |P_{i,j}(T,(3n-1)\times (3n-1))| }\\
	&=& & |P_{1,1}(T,(3n)\times (3n))| + |P_{1,1}(T,(3n)\times (3n))| \\ 
	& &+& |P_{1,1}(T,(3n)\times (3n+3))| + |P_{1,1}(T,(3n)\times (3n))| \\ 
	& &+& |P_{1,1}(T,(3n)\times (3n))| + |P_{1,1}(T,(3n)\times (3n+3))| \\
	& &+& |P_{1,1}(T,(3n+3)\times (3n))| + |P_{1,1}(T,(3n+3)\times (3n))| \\
	& &+& |P_{1,1}(T,(3n+3)\times (3n+3))| \\
	&=& & A_{n} + A_{n} + B_{n}  \\
	& &+& A_{n} + A_{n} + B_{n}  \\
	& &+& C_{n} + C_{n} + A_{n+1},
\end{array}$}
\end{equation}
see also Figure~\ref{fig:recA3n1}.

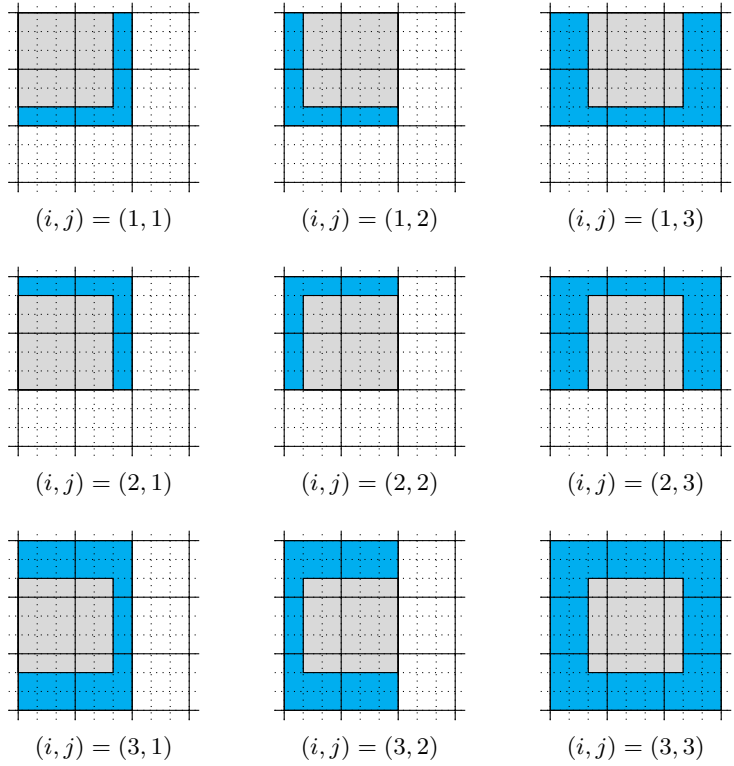
\begin{figure}[ht]
	\begin{center}
		\begin{tikzpicture}[scale = 0.25]
			\begin{scope}[xshift = 0cm, yshift = 28cm]			
				\draw[color=black, fill=gray, fill opacity=0.3] (0, 4) rectangle ++ (5, 5);
				\draw[color=black, fill=cyan, even odd rule] (0, 4) rectangle ++ (5, 5) (0, 3) rectangle ++ (6, 6);
				\draw[step=1, black, dotted]    (-0.5,-0.5) grid (9.5, 9.5);
				\draw[step=3, black, cap=rect]  (-0.5,-0.5) grid (9.5, 9.5);
				\draw(4.5, -2) node{\footnotesize $(i,j) = (1,1)$};
			\end{scope}
			\begin{scope}[xshift = 14cm, yshift = 28cm]
				\draw[color=black, fill=gray, fill opacity=0.3] (1, 4) rectangle ++ (5, 5);
				\draw[color=black, fill=cyan, even odd rule] (1, 4) rectangle ++ (5, 5) (0, 3) rectangle ++ (6, 6);
				\draw[step=1, black, dotted]    (-0.5,-0.5) grid (9.5, 9.5);
				\draw[step=3, black, cap=rect]  (-0.5,-0.5) grid (9.5, 9.5);
				\draw(4.5, -2) node{\footnotesize $(i,j) = (1,2)$};
			\end{scope}
			\begin{scope}[xshift = 28cm, yshift = 28cm]
				\draw[color=black, fill=gray, fill opacity=0.3] (2, 4) rectangle ++ (5, 5);
				\draw[color=black, fill=cyan, even odd rule] (2, 4) rectangle ++ (5, 5) (0, 3) rectangle ++ (9, 6);
				\draw[step=1, black, dotted]    (-0.5,-0.5) grid (9.5, 9.5);
				\draw[step=3, black, cap=rect]  (-0.5,-0.5) grid (9.5, 9.5);
				\draw(4.5, -2) node{\footnotesize $(i,j) = (1,3)$};
			\end{scope}
			\begin{scope}[xshift = 0cm, yshift = 14cm]			
				\draw[color=black, fill=gray, fill opacity=0.3] (0, 3) rectangle ++ (5, 5);
				\draw[color=black, fill=cyan, even odd rule] (0, 3) rectangle ++ (5, 5) (0, 3) rectangle ++ (6, 6);
				\draw[step=1, black, dotted]    (-0.5,-0.5) grid (9.5, 9.5);
				\draw[step=3, black, cap=rect]  (-0.5,-0.5) grid (9.5, 9.5);
				\draw(4.5, -2) node{\footnotesize $(i,j) = (2,1)$};
			\end{scope}
			\begin{scope}[xshift = 14cm, yshift = 14cm]
				\draw[color=black, fill=gray, fill opacity=0.3] (1, 3) rectangle ++ (5, 5);
				\draw[color=black, fill=cyan, even odd rule] (1, 3) rectangle ++ (5, 5) (0, 3) rectangle ++ (6, 6);
				\draw[step=1, black, dotted]    (-0.5,-0.5) grid (9.5, 9.5);
				\draw[step=3, black, cap=rect]  (-0.5,-0.5) grid (9.5, 9.5);
				\draw(4.5, -2) node{\footnotesize $(i,j) = (2,2)$};
			\end{scope}
			\begin{scope}[xshift = 28cm, yshift = 14cm]
				\draw[color=black, fill=gray, fill opacity=0.3] (2, 3) rectangle ++ (5, 5);
				\draw[color=black, fill=cyan, even odd rule] (2, 3) rectangle ++ (5, 5) (0, 3) rectangle ++ (9, 6);
				\draw[step=1, black, dotted]    (-0.5,-0.5) grid (9.5, 9.5);
				\draw[step=3, black, cap=rect]  (-0.5,-0.5) grid (9.5, 9.5);
				\draw(4.5, -2) node{\footnotesize $(i,j) = (2,3)$};
			\end{scope}
			\begin{scope}[xshift = 0cm, yshift = 0cm]			
				\draw[color=black, fill=gray, fill opacity=0.3] (0, 2) rectangle ++ (5, 5);
				\draw[color=black, fill=cyan, even odd rule] (0, 2) rectangle ++ (5, 5) (0, 0) rectangle ++ (6, 9);
				\draw[step=1, black, dotted]    (-0.5,-0.5) grid (9.5, 9.5);
				\draw[step=3, black, cap=rect]  (-0.5,-0.5) grid (9.5, 9.5);
				\draw(4.5, -2) node{\footnotesize $(i,j) = (3,1)$};
			\end{scope}
			\begin{scope}[xshift = 14cm, yshift = 0cm]
				\draw[color=black, fill=gray, fill opacity=0.3] (1, 2) rectangle ++ (5, 5);
				\draw[color=black, fill=cyan, even odd rule] (1, 2) rectangle ++ (5, 5) (0, 0) rectangle ++ (6, 9);
				\draw[step=1, black, dotted]    (-0.5,-0.5) grid (9.5, 9.5);
				\draw[step=3, black, cap=rect]  (-0.5,-0.5) grid (9.5, 9.5);
				\draw(4.5, -2) node{\footnotesize $(i,j) = (3,2)$};
			\end{scope}
			\begin{scope}[xshift = 28cm, yshift = 0cm]
				\draw[color=black, fill=gray, fill opacity=0.3] (2, 2) rectangle ++ (5, 5);
				\draw[color=black, fill=cyan, even odd rule] (2, 2) rectangle ++ (5, 5) (0, 0) rectangle ++ (9, 9);
				\draw[step=1, black, dotted]    (-0.5,-0.5) grid (9.5, 9.5);
				\draw[step=3, black, cap=rect]  (-0.5,-0.5) grid (9.5, 9.5);
				\draw(4.5, -2) node{\footnotesize $(i,j) = (3,3)$};
			\end{scope}
		\end{tikzpicture}
		\caption{Extensions used in \eqref{eq:recA3n1} for the recursion of $A_{3n-1}$. The gray region represent an element in $P_{i,j}(T,(3n-1)\times(3n-1))$, and the blue one the region it is extend with.}
     \label{fig:recA3n1}
	\end{center}
\end{figure}

\begin{equation}
\label{eq:recA3n}
\renewcommand{\arraystretch}{1.5}
\text{\small$\begin{array}{c@{\ }c@{\ }c@{\ }l}
	A_{3n} 
	&=& \multicolumn{2}{l}{ \displaystyle{\sum_{i=1}^3 \sum_{j=1}^3} |P_{i,j}(T,(3n)\times (3n))| }\\
	&=& & |P_{1,1}(T,(3n)\times (3n))| + |P_{1,1}(T,(3n)\times (3n+3))| \\ 
	& &+& |P_{1,1}(T,(3n)\times (3n+3))| + |P_{1,1}(T,(3n)\times (3n+3))| \\ 
	& &+& |P_{1,1}(T,(3n+3)\times (3n+3))| + |P_{1,1}(T,(3n)\times (3n+3))| \\
	& &+& |P_{1,1}(T,(3n+3)\times (3n))| + |P_{1,1}(T,(3n+3)\times (3n+3))| \\
	& &+& |P_{1,1}(T,(3n+3)\times (3n+3))| \\
	&=& & A_{n} + B_{n} + B_{n}  \\
	& &+& C_{n} + A_{n+1} + A_{n+1}  \\
	& &+& C_{n} + A_{n+1} + A_{n+1}, 
\end{array}$}
\end{equation}
see also Figure~\ref{fig:recA3n}.

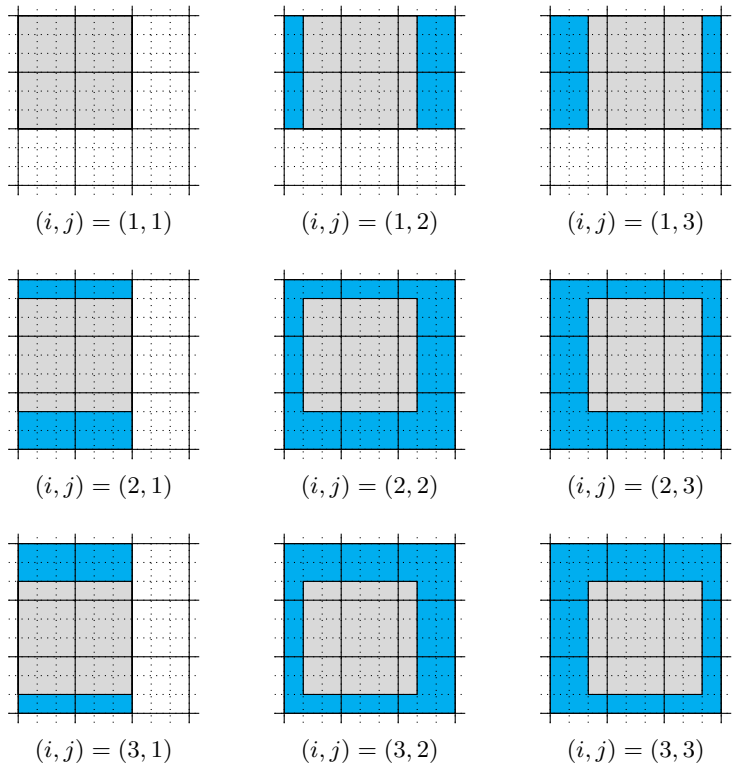
\begin{figure}[ht]
	\begin{center}
		\begin{tikzpicture}[scale = 0.25]
			\begin{scope}[xshift = 0cm, yshift = 28cm]			
				\draw[color=black, fill=gray, fill opacity=0.3] (0, 3) rectangle ++ (6, 6);
				\draw[color=black, fill=cyan, even odd rule] (0, 3) rectangle ++ (6, 6) (0, 3) rectangle ++ (6, 6);
				\draw[step=1, black, dotted]    (-0.5,-0.5) grid (9.5, 9.5);
				\draw[step=3, black, cap=rect]  (-0.5,-0.5) grid (9.5, 9.5);
				\draw(4.5, -2) node{\footnotesize $(i,j) = (1,1)$};
			\end{scope}
			\begin{scope}[xshift = 14cm, yshift = 28cm]
				\draw[color=black, fill=gray, fill opacity=0.3] (1, 3) rectangle ++ (6, 6);
				\draw[color=black, fill=cyan, even odd rule] (1, 3) rectangle ++ (6, 6) (0, 3) rectangle ++ (9, 6);
				\draw[step=1, black, dotted]    (-0.5,-0.5) grid (9.5, 9.5);
				\draw[step=3, black, cap=rect]  (-0.5,-0.5) grid (9.5, 9.5);
				\draw(4.5, -2) node{\footnotesize $(i,j) = (1,2)$};
			\end{scope}
			\begin{scope}[xshift = 28cm, yshift = 28cm]
				\draw[color=black, fill=gray, fill opacity=0.3] (2, 3) rectangle ++ (6, 6);
				\draw[color=black, fill=cyan, even odd rule] (2, 3) rectangle ++ (6, 6) (0, 3) rectangle ++ (9, 6);
				\draw[step=1, black, dotted]    (-0.5,-0.5) grid (9.5, 9.5);
				\draw[step=3, black, cap=rect]  (-0.5,-0.5) grid (9.5, 9.5);
				\draw(4.5, -2) node{\footnotesize $(i,j) = (1,3)$};
			\end{scope}
			\begin{scope}[xshift = 0cm, yshift = 14cm]			
				\draw[color=black, fill=gray, fill opacity=0.3] (0, 2) rectangle ++ (6, 6);
				\draw[color=black, fill=cyan, even odd rule] (0, 2) rectangle ++ (6, 6) (0, 0) rectangle ++ (6, 9);
				\draw[step=1, black, dotted]    (-0.5,-0.5) grid (9.5, 9.5);
				\draw[step=3, black, cap=rect]  (-0.5,-0.5) grid (9.5, 9.5);
				\draw(4.5, -2) node{\footnotesize $(i,j) = (2,1)$};
			\end{scope}
			\begin{scope}[xshift = 14cm, yshift = 14cm]
				\draw[color=black, fill=gray, fill opacity=0.3] (1, 2) rectangle ++ (6, 6);
				\draw[color=black, fill=cyan, even odd rule] (1, 2) rectangle ++ (6, 6) (0, 0) rectangle ++ (9, 9);
				\draw[step=1, black, dotted]    (-0.5,-0.5) grid (9.5, 9.5);
				\draw[step=3, black, cap=rect]  (-0.5,-0.5) grid (9.5, 9.5);
				\draw(4.5, -2) node{\footnotesize $(i,j) = (2,2)$};
			\end{scope}
			\begin{scope}[xshift = 28cm, yshift = 14cm]
				\draw[color=black, fill=gray, fill opacity=0.3] (2, 2) rectangle ++ (6, 6);
				\draw[color=black, fill=cyan, even odd rule] (2, 2) rectangle ++ (6, 6) (0, 0) rectangle ++ (9, 9);
				\draw[step=1, black, dotted]    (-0.5,-0.5) grid (9.5, 9.5);
				\draw[step=3, black, cap=rect]  (-0.5,-0.5) grid (9.5, 9.5);
				\draw(4.5, -2) node{\footnotesize $(i,j) = (2,3)$};
			\end{scope}
			\begin{scope}[xshift = 0cm, yshift = 0cm]			
				\draw[color=black, fill=gray, fill opacity=0.3] (0, 1) rectangle ++ (6, 6);
				\draw[color=black, fill=cyan, even odd rule] (0, 1) rectangle ++ (6, 6) (0, 0) rectangle ++ (6, 9);
				\draw[step=1, black, dotted]    (-0.5,-0.5) grid (9.5, 9.5);
				\draw[step=3, black, cap=rect]  (-0.5,-0.5) grid (9.5, 9.5);
				\draw(4.5, -2) node{\footnotesize $(i,j) = (3,1)$};
			\end{scope}
			\begin{scope}[xshift = 14cm, yshift = 0cm]
				\draw[color=black, fill=gray, fill opacity=0.3] (1, 1) rectangle ++ (6, 6);
				\draw[color=black, fill=cyan, even odd rule] (1, 1) rectangle ++ (6, 6) (0, 0) rectangle ++ (9, 9);
				\draw[step=1, black, dotted]    (-0.5,-0.5) grid (9.5, 9.5);
				\draw[step=3, black, cap=rect]  (-0.5,-0.5) grid (9.5, 9.5);
				\draw(4.5, -2) node{\footnotesize $(i,j) = (3,2)$};
			\end{scope}
			\begin{scope}[xshift = 28cm, yshift = 0cm]
				\draw[color=black, fill=gray, fill opacity=0.3] (2, 1) rectangle ++ (6, 6);
				\draw[color=black, fill=cyan, even odd rule] (2, 1) rectangle ++ (6, 6) (0, 0) rectangle ++ (9, 9);
				\draw[step=1, black, dotted]    (-0.5,-0.5) grid (9.5, 9.5);
				\draw[step=3, black, cap=rect]  (-0.5,-0.5) grid (9.5, 9.5);
				\draw(4.5, -2) node{\footnotesize $(i,j) = (3,3)$};
			\end{scope}
		\end{tikzpicture}
		\caption{Extensions used in \eqref{eq:recA3n} for the recursion of $A_{3n}$. The gray region represent an element in $P_{i,j}(T,(3n)\times(3n))$, and the blue one the region it is extend with.}
     \label{fig:recA3n}
	\end{center}
\end{figure}

\begin{equation}
\label{eq:recB3n2}
\renewcommand{\arraystretch}{1.5}
\text{\small$\begin{array}{c@{\ }c@{\ }c@{\ }l}
	B_{3n-2} 
	&=& \multicolumn{2}{l}{ \displaystyle{\sum_{i=1}^3 \sum_{j=1}^3} |P_{i,j}(T,(3n-2)\times (3n-1))| }\\
	&=& & |P_{1,1}(T,(3n)\times (3n))| + |P_{1,1}(T,(3n)\times (3n))| \\ 
	& &+& |P_{1,1}(T,(3n)\times (3n+3))| + |P_{1,1}(T,(3n)\times (3n))| \\ 
	& &+& |P_{1,1}(T,(3n)\times (3n))| + |P_{1,1}(T,(3n)\times (3n+3))| \\
	& &+& |P_{1,1}(T,(3n)\times (3n))| + |P_{1,1}(T,(3n)\times (3n))| \\
	& &+& |P_{1,1}(T,(3n)\times (3n+3))| \\
	&=& & A_{n} + A_{n} + B_{n}  \\
	& &+& A_{n} + A_{n} + B_{n}  \\
	& &+& A_{n} + A_{n} + B_{n}, 
\end{array}$}
\end{equation}
see also Figure~\ref{fig:recB3n2}.

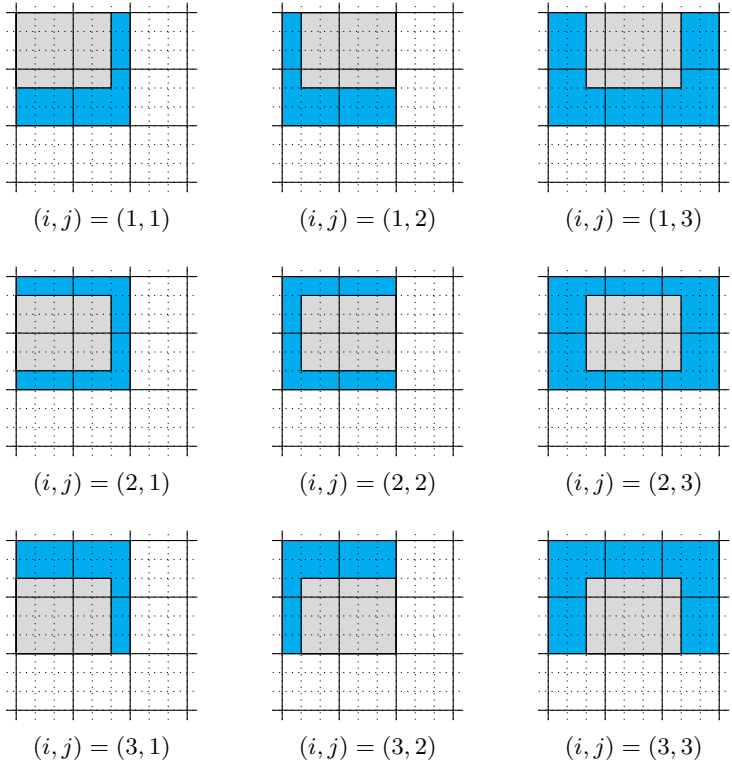
\begin{figure}[ht]
	\begin{center}
		\begin{tikzpicture}[scale = 0.25]
			\begin{scope}[xshift = 0cm, yshift = 28cm]			
				\draw[color=black, fill=gray, fill opacity=0.3] (0, 5) rectangle ++ (5, 4);
				\draw[color=black, fill=cyan, even odd rule] (0, 5) rectangle ++ (5, 4) (0, 3) rectangle ++ (6, 6);
				\draw[step=1, black, dotted]    (-0.5,-0.5) grid (9.5, 9.5);
				\draw[step=3, black, cap=rect]  (-0.5,-0.5) grid (9.5, 9.5);
				\draw(4.5, -2) node{\footnotesize $(i,j) = (1,1)$};
			\end{scope}
			\begin{scope}[xshift = 14cm, yshift = 28cm]
				\draw[color=black, fill=gray, fill opacity=0.3] (1, 5) rectangle ++ (5, 4);
				\draw[color=black, fill=cyan, even odd rule] (1, 5) rectangle ++ (5, 4) (0, 3) rectangle ++ (6, 6);
				\draw[step=1, black, dotted]    (-0.5,-0.5) grid (9.5, 9.5);
				\draw[step=3, black, cap=rect]  (-0.5,-0.5) grid (9.5, 9.5);
				\draw(4.5, -2) node{\footnotesize $(i,j) = (1,2)$};
			\end{scope}
			\begin{scope}[xshift = 28cm, yshift = 28cm]
				\draw[color=black, fill=gray, fill opacity=0.3] (2, 5) rectangle ++ (5, 4);
				\draw[color=black, fill=cyan, even odd rule] (2, 5) rectangle ++ (5, 4) (0, 3) rectangle ++ (9, 6);
				\draw[step=1, black, dotted]    (-0.5,-0.5) grid (9.5, 9.5);
				\draw[step=3, black, cap=rect]  (-0.5,-0.5) grid (9.5, 9.5);
				\draw(4.5, -2) node{\footnotesize $(i,j) = (1,3)$};
			\end{scope}
			\begin{scope}[xshift = 0cm, yshift = 14cm]			
				\draw[color=black, fill=gray, fill opacity=0.3] (0, 4) rectangle ++ (5, 4);
				\draw[color=black, fill=cyan, even odd rule] (0, 4) rectangle ++ (5, 4) (0, 3) rectangle ++ (6, 6);
				\draw[step=1, black, dotted]    (-0.5,-0.5) grid (9.5, 9.5);
				\draw[step=3, black, cap=rect]  (-0.5,-0.5) grid (9.5, 9.5);
				\draw(4.5, -2) node{\footnotesize $(i,j) = (2,1)$};
			\end{scope}
			\begin{scope}[xshift = 14cm, yshift = 14cm]
				\draw[color=black, fill=gray, fill opacity=0.3] (1, 4) rectangle ++ (5, 4);
				\draw[color=black, fill=cyan, even odd rule] (1, 4) rectangle ++ (5, 4) (0, 3) rectangle ++ (6, 6);
				\draw[step=1, black, dotted]    (-0.5,-0.5) grid (9.5, 9.5);
				\draw[step=3, black, cap=rect]  (-0.5,-0.5) grid (9.5, 9.5);
				\draw(4.5, -2) node{\footnotesize $(i,j) = (2,2)$};
			\end{scope}
			\begin{scope}[xshift = 28cm, yshift = 14cm]
				\draw[color=black, fill=gray, fill opacity=0.3] (2, 4) rectangle ++ (5, 4);
				\draw[color=black, fill=cyan, even odd rule] (2, 4) rectangle ++ (5, 4) (0, 3) rectangle ++ (9, 6);
				\draw[step=1, black, dotted]    (-0.5,-0.5) grid (9.5, 9.5);
				\draw[step=3, black, cap=rect]  (-0.5,-0.5) grid (9.5, 9.5);
				\draw(4.5, -2) node{\footnotesize $(i,j) = (2,3)$};
			\end{scope}
			\begin{scope}[xshift = 0cm, yshift = 0cm]			
				\draw[color=black, fill=gray, fill opacity=0.3] (0, 3) rectangle ++ (5, 4);
				\draw[color=black, fill=cyan, even odd rule] (0, 3) rectangle ++ (5, 4) (0, 3) rectangle ++ (6, 6);
				\draw[step=1, black, dotted]    (-0.5,-0.5) grid (9.5, 9.5);
				\draw[step=3, black, cap=rect]  (-0.5,-0.5) grid (9.5, 9.5);
				\draw(4.5, -2) node{\footnotesize $(i,j) = (3,1)$};
			\end{scope}
			\begin{scope}[xshift = 14cm, yshift = 0cm]
				\draw[color=black, fill=gray, fill opacity=0.3] (1, 3) rectangle ++ (5, 4);
				\draw[color=black, fill=cyan, even odd rule] (1, 3) rectangle ++ (5, 4) (0, 3) rectangle ++ (6, 6);
				\draw[step=1, black, dotted]    (-0.5,-0.5) grid (9.5, 9.5);
				\draw[step=3, black, cap=rect]  (-0.5,-0.5) grid (9.5, 9.5);
				\draw(4.5, -2) node{\footnotesize $(i,j) = (3,2)$};
			\end{scope}
			\begin{scope}[xshift = 28cm, yshift = 0cm]
				\draw[color=black, fill=gray, fill opacity=0.3] (2, 3) rectangle ++ (5, 4);
				\draw[color=black, fill=cyan, even odd rule] (2, 3) rectangle ++ (5, 4) (0, 3) rectangle ++ (9, 6);
				\draw[step=1, black, dotted]    (-0.5,-0.5) grid (9.5, 9.5);
				\draw[step=3, black, cap=rect]  (-0.5,-0.5) grid (9.5, 9.5);
				\draw(4.5, -2) node{\footnotesize $(i,j) = (3,3)$};
			\end{scope}
		\end{tikzpicture}
		\caption{Extensions used in \eqref{eq:recB3n2} for the recursion of $B_{3n-2}$. The gray region represent an element in $P_{i,j}(T,(3n-2)\times(3n-1))$, and the blue one the region it is extend with.}
     \label{fig:recB3n2}
	\end{center}
\end{figure}

\begin{equation}
\label{eq:recB3n1}
\renewcommand{\arraystretch}{1.5}
\text{\small$\begin{array}{c@{\ }c@{\ }c@{\ }l}
	B_{3n-1} 
	&=& \multicolumn{2}{l}{ \displaystyle{\sum_{i=1}^3 \sum_{j=1}^3} |P_{i,j}(T,(3n-1)\times (3n))| }\\
	&=& & |P_{1,1}(T,(3n)\times (3n))| + |P_{1,1}(T,(3n)\times (3n+3))| \\ 
	& &+& |P_{1,1}(T,(3n)\times (3n+3))| + |P_{1,1}(T,(3n)\times (3n))| \\ 
	& &+& |P_{1,1}(T,(3n)\times (3n+3))| + |P_{1,1}(T,(3n)\times (3n+3))| \\
	& &+& |P_{1,1}(T,(3n+3)\times (3n))| + |P_{1,1}(T,(3n+3)\times (3n+3))| \\
	& &+& |P_{1,1}(T,(3n+3)\times (3n+3))| \\
	&=& & A_{n} + B_{n} + B_{n}  \\
	& &+& A_{n} + B_{n} + B_{n}  \\
	& &+& C_{n} + A_{n+1} + A_{n+1}, 
\end{array}$}
\end{equation}
see also Figure~\ref{fig:recB3n1}.

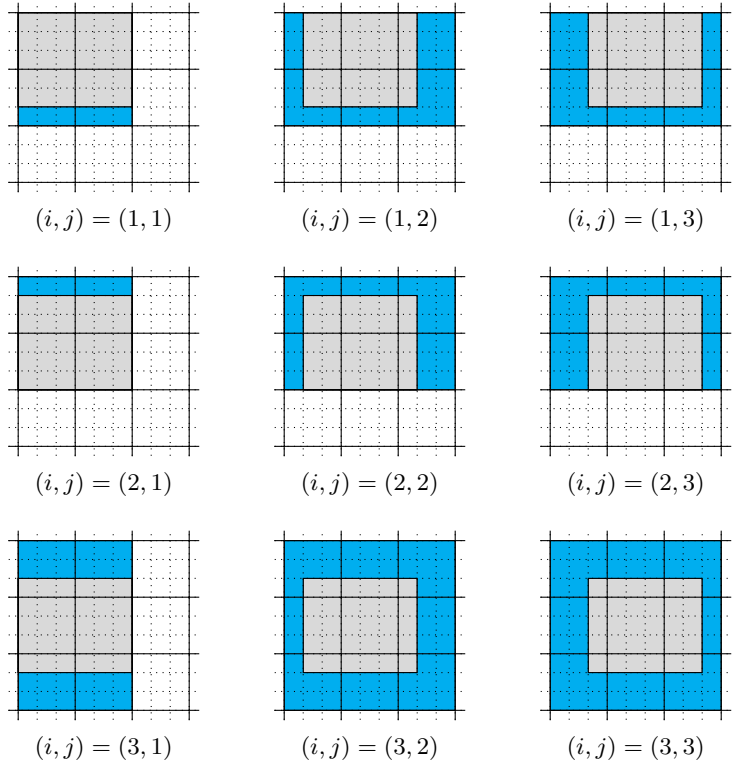
\begin{figure}[ht]
	\begin{center}
		\begin{tikzpicture}[scale = 0.25]
			\begin{scope}[xshift = 0cm, yshift = 28cm]			
				\draw[color=black, fill=gray, fill opacity=0.3] (0, 4) rectangle ++ (6, 5);
				\draw[color=black, fill=cyan, even odd rule] (0, 4) rectangle ++ (6, 5) (0, 3) rectangle ++ (6, 6);
				\draw[step=1, black, dotted]    (-0.5,-0.5) grid (9.5, 9.5);
				\draw[step=3, black, cap=rect]  (-0.5,-0.5) grid (9.5, 9.5);
				\draw(4.5, -2) node{\footnotesize $(i,j) = (1,1)$};
			\end{scope}
			\begin{scope}[xshift = 14cm, yshift = 28cm]
				\draw[color=black, fill=gray, fill opacity=0.3] (1, 4) rectangle ++ (6, 5);
				\draw[color=black, fill=cyan, even odd rule] (1, 4) rectangle ++ (6, 5) (0, 3) rectangle ++ (9, 6);
				\draw[step=1, black, dotted]    (-0.5,-0.5) grid (9.5, 9.5);
				\draw[step=3, black, cap=rect]  (-0.5,-0.5) grid (9.5, 9.5);
				\draw(4.5, -2) node{\footnotesize $(i,j) = (1,2)$};
			\end{scope}
			\begin{scope}[xshift = 28cm, yshift = 28cm]
				\draw[color=black, fill=gray, fill opacity=0.3] (2, 4) rectangle ++ (6, 5);
				\draw[color=black, fill=cyan, even odd rule] (2, 4) rectangle ++ (6, 5) (0, 3) rectangle ++ (9, 6);
				\draw[step=1, black, dotted]    (-0.5,-0.5) grid (9.5, 9.5);
				\draw[step=3, black, cap=rect]  (-0.5,-0.5) grid (9.5, 9.5);
				\draw(4.5, -2) node{\footnotesize $(i,j) = (1,3)$};
			\end{scope}
			\begin{scope}[xshift = 0cm, yshift = 14cm]			
				\draw[color=black, fill=gray, fill opacity=0.3] (0, 3) rectangle ++ (6, 5);
				\draw[color=black, fill=cyan, even odd rule] (0, 3) rectangle ++ (6, 5) (0, 3) rectangle ++ (6, 6);
				\draw[step=1, black, dotted]    (-0.5,-0.5) grid (9.5, 9.5);
				\draw[step=3, black, cap=rect]  (-0.5,-0.5) grid (9.5, 9.5);
				\draw(4.5, -2) node{\footnotesize $(i,j) = (2,1)$};
			\end{scope}
			\begin{scope}[xshift = 14cm, yshift = 14cm]
				\draw[color=black, fill=gray, fill opacity=0.3] (1, 3) rectangle ++ (6, 5);
				\draw[color=black, fill=cyan, even odd rule] (1, 3) rectangle ++ (6, 5) (0, 3) rectangle ++ (9, 6);
				\draw[step=1, black, dotted]    (-0.5,-0.5) grid (9.5, 9.5);
				\draw[step=3, black, cap=rect]  (-0.5,-0.5) grid (9.5, 9.5);
				\draw(4.5, -2) node{\footnotesize $(i,j) = (2,2)$};
			\end{scope}
			\begin{scope}[xshift = 28cm, yshift = 14cm]
				\draw[color=black, fill=gray, fill opacity=0.3] (2, 3) rectangle ++ (6, 5);
				\draw[color=black, fill=cyan, even odd rule] (2, 3) rectangle ++ (6, 5) (0, 3) rectangle ++ (9, 6);
				\draw[step=1, black, dotted]    (-0.5,-0.5) grid (9.5, 9.5);
				\draw[step=3, black, cap=rect]  (-0.5,-0.5) grid (9.5, 9.5);
				\draw(4.5, -2) node{\footnotesize $(i,j) = (2,3)$};
			\end{scope}
			\begin{scope}[xshift = 0cm, yshift = 0cm]			
				\draw[color=black, fill=gray, fill opacity=0.3] (0, 2) rectangle ++ (6, 5);
				\draw[color=black, fill=cyan, even odd rule] (0, 2) rectangle ++ (6, 5) (0, 0) rectangle ++ (6, 9);
				\draw[step=1, black, dotted]    (-0.5,-0.5) grid (9.5, 9.5);
				\draw[step=3, black, cap=rect]  (-0.5,-0.5) grid (9.5, 9.5);
				\draw(4.5, -2) node{\footnotesize $(i,j) = (3,1)$};
			\end{scope}
			\begin{scope}[xshift = 14cm, yshift = 0cm]
				\draw[color=black, fill=gray, fill opacity=0.3] (1, 2) rectangle ++ (6, 5);
				\draw[color=black, fill=cyan, even odd rule] (1, 2) rectangle ++ (6, 5) (0, 0) rectangle ++ (9, 9);
				\draw[step=1, black, dotted]    (-0.5,-0.5) grid (9.5, 9.5);
				\draw[step=3, black, cap=rect]  (-0.5,-0.5) grid (9.5, 9.5);
				\draw(4.5, -2) node{\footnotesize $(i,j) = (3,2)$};
			\end{scope}
			\begin{scope}[xshift = 28cm, yshift = 0cm]
				\draw[color=black, fill=gray, fill opacity=0.3] (2, 2) rectangle ++ (6, 5);
				\draw[color=black, fill=cyan, even odd rule] (2, 2) rectangle ++ (6, 5) (0, 0) rectangle ++ (9, 9);
				\draw[step=1, black, dotted]    (-0.5,-0.5) grid (9.5, 9.5);
				\draw[step=3, black, cap=rect]  (-0.5,-0.5) grid (9.5, 9.5);
				\draw(4.5, -2) node{\footnotesize $(i,j) = (3,3)$};
			\end{scope}
		\end{tikzpicture}
		\caption{Extensions used in \eqref{eq:recB3n1} for the recursion of $B_{3n-1}$. The gray region represent an element in $P_{i,j}(T,(3n-1)\times(3n))$, and the blue one the region it is extend with.}
     \label{fig:recB3n1}
	\end{center}
\end{figure}

\begin{equation}
\label{eq:recB3n}
\renewcommand{\arraystretch}{1.5}
\text{\small$\begin{array}{c@{\ }c@{\ }c@{\ }l}
	B_{3n} 
	&=& \multicolumn{2}{l}{ \displaystyle{\sum_{i=1}^3 \sum_{j=1}^3} |P_{i,j}(T,(3n)\times (3n+1))| }\\
	&=& & |P_{1,1}(T,(3n)\times (3n+3))| + |P_{1,1}(T,(3n)\times (3n+3))| \\ 
	& &+& |P_{1,1}(T,(3n)\times (3n+3))| + |P_{1,1}(T,(3n+3)\times (3n+3))| \\ 
	& &+& |P_{1,1}(T,(3n+3)\times (3n+3))| + |P_{1,1}(T,(3n+3)\times (3n+3))| \\
	& &+& |P_{1,1}(T,(3n+3)\times (3n+3))| + |P_{1,1}(T,(3n+3)\times (3n+3))| \\
	& &+& |P_{1,1}(T,(3n+3)\times (3n+3))| \\
	&=& & B_{n}   + B_{n}   + B_{n}  \\
	& &+& A_{n+1} + A_{n+1} + A_{n+1}  \\
	& &+& A_{n+1} + A_{n+1} + A_{n+1}, 
\end{array}$}
\end{equation}
see also Figure~\ref{fig:recB3n}.

\begin{figure}[ht]
	\begin{center}
		\begin{tikzpicture}[scale = 0.25]
			\begin{scope}[xshift = 0cm, yshift = 28cm]			
				\draw[color=black, fill=gray, fill opacity=0.3] (0, 3) rectangle ++ (7, 6);
				\draw[color=black, fill=cyan, even odd rule] (0, 3) rectangle ++ (7, 6) (0, 3) rectangle ++ (9, 6);
				\draw[step=1, black, dotted]    (-0.5,-0.5) grid (9.5, 9.5);
				\draw[step=3, black, cap=rect]  (-0.5,-0.5) grid (9.5, 9.5);
				\draw(4.5, -2) node{\footnotesize $(i,j) = (1,1)$};
			\end{scope}
			\begin{scope}[xshift = 14cm, yshift = 28cm]
				\draw[color=black, fill=gray, fill opacity=0.3] (1, 3) rectangle ++ (7, 6);
				\draw[color=black, fill=cyan, even odd rule] (1, 3) rectangle ++ (7, 6) (0, 3) rectangle ++ (9, 6);
				\draw[step=1, black, dotted]    (-0.5,-0.5) grid (9.5, 9.5);
				\draw[step=3, black, cap=rect]  (-0.5,-0.5) grid (9.5, 9.5);
				\draw(4.5, -2) node{\footnotesize $(i,j) = (1,2)$};
			\end{scope}
			\begin{scope}[xshift = 28cm, yshift = 28cm]
				\draw[color=black, fill=gray, fill opacity=0.3] (2, 3) rectangle ++ (7, 6);
				\draw[color=black, fill=cyan, even odd rule] (2, 3) rectangle ++ (7, 6) (0, 3) rectangle ++ (9, 6);
				\draw[step=1, black, dotted]    (-0.5,-0.5) grid (9.5, 9.5);
				\draw[step=3, black, cap=rect]  (-0.5,-0.5) grid (9.5, 9.5);
				\draw(4.5, -2) node{\footnotesize $(i,j) = (1,3)$};
			\end{scope}
			\begin{scope}[xshift = 0cm, yshift = 14cm]			
				\draw[color=black, fill=gray, fill opacity=0.3] (0, 2) rectangle ++ (7, 6);
				\draw[color=black, fill=cyan, even odd rule] (0, 2) rectangle ++ (7, 6) (0, 0) rectangle ++ (9, 9);
				\draw[step=1, black, dotted]    (-0.5,-0.5) grid (9.5, 9.5);
				\draw[step=3, black, cap=rect]  (-0.5,-0.5) grid (9.5, 9.5);
				\draw(4.5, -2) node{\footnotesize $(i,j) = (2,1)$};
			\end{scope}
			\begin{scope}[xshift = 14cm, yshift = 14cm]
				\draw[color=black, fill=gray, fill opacity=0.3] (1, 2) rectangle ++ (7, 6);
				\draw[color=black, fill=cyan, even odd rule] (1, 2) rectangle ++ (7, 6) (0, 0) rectangle ++ (9, 9);
				\draw[step=1, black, dotted]    (-0.5,-0.5) grid (9.5, 9.5);
				\draw[step=3, black, cap=rect]  (-0.5,-0.5) grid (9.5, 9.5);
				\draw(4.5, -2) node{\footnotesize $(i,j) = (2,2)$};
			\end{scope}
			\begin{scope}[xshift = 28cm, yshift = 14cm]
				\draw[color=black, fill=gray, fill opacity=0.3] (2, 2) rectangle ++ (7, 6);
				\draw[color=black, fill=cyan, even odd rule] (2, 2) rectangle ++ (7, 6) (0, 0) rectangle ++ (9, 9);
				\draw[step=1, black, dotted]    (-0.5,-0.5) grid (9.5, 9.5);
				\draw[step=3, black, cap=rect]  (-0.5,-0.5) grid (9.5, 9.5);
				\draw(4.5, -2) node{\footnotesize $(i,j) = (2,3)$};
			\end{scope}
			\begin{scope}[xshift = 0cm, yshift = 0cm]			
				\draw[color=black, fill=gray, fill opacity=0.3] (0, 1) rectangle ++ (7, 6);
				\draw[color=black, fill=cyan, even odd rule] (0, 1) rectangle ++ (7, 6) (0, 0) rectangle ++ (9, 9);
				\draw[step=1, black, dotted]    (-0.5,-0.5) grid (9.5, 9.5);
				\draw[step=3, black, cap=rect]  (-0.5,-0.5) grid (9.5, 9.5);
				\draw(4.5, -2) node{\footnotesize $(i,j) = (3,1)$};
			\end{scope}
			\begin{scope}[xshift = 14cm, yshift = 0cm]
				\draw[color=black, fill=gray, fill opacity=0.3] (1, 1) rectangle ++ (7, 6);
				\draw[color=black, fill=cyan, even odd rule] (1, 1) rectangle ++ (7, 6) (0, 0) rectangle ++ (9, 9);
				\draw[step=1, black, dotted]    (-0.5,-0.5) grid (9.5, 9.5);
				\draw[step=3, black, cap=rect]  (-0.5,-0.5) grid (9.5, 9.5);
				\draw(4.5, -2) node{\footnotesize $(i,j) = (3,2)$};
			\end{scope}
			\begin{scope}[xshift = 28cm, yshift = 0cm]
				\draw[color=black, fill=gray, fill opacity=0.3] (2, 1) rectangle ++ (7, 6);
				\draw[color=black, fill=cyan, even odd rule] (2, 1) rectangle ++ (7, 6) (0, 0) rectangle ++ (9, 9);
				\draw[step=1, black, dotted]    (-0.5,-0.5) grid (9.5, 9.5);
				\draw[step=3, black, cap=rect]  (-0.5,-0.5) grid (9.5, 9.5);
				\draw(4.5, -2) node{\footnotesize $(i,j) = (3,3)$};
			\end{scope}
		\end{tikzpicture}
		\caption{Extensions used in \eqref{eq:recB3n} for the recursion of $B_{3n}$. The gray region represent an element in $P_{i,j}(T,(3n)\times(3n+1))$, and the blue one the region it is extend with.}
     \label{fig:recB3n}
	\end{center}
\end{figure}
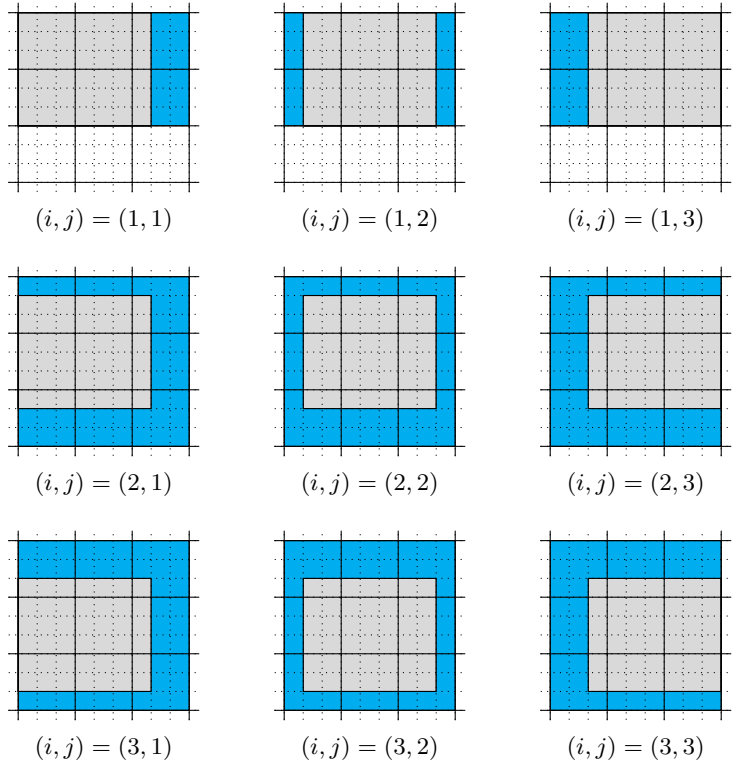

\begin{equation}
\label{eq:recC3n2}
\renewcommand{\arraystretch}{1.5}
\text{\small$\begin{array}{c@{\ }c@{\ }c@{\ }l}
	C_{3n-2} 
	&=& \multicolumn{2}{l}{ \displaystyle{\sum_{i=1}^3 \sum_{j=1}^3 } |P_{i,j}(T,(3n-1)\times (3n-2))| }\\
	&=& & |P_{1,1}(T,(3n)\times (3n))| + |P_{1,1}(T,(3n)\times (3n))| \\ 
	& &+& |P_{1,1}(T,(3n)\times (3n))| + |P_{1,1}(T,(3n)\times (3n))| \\ 
	& &+& |P_{1,1}(T,(3n)\times (3n))| + |P_{1,1}(T,(3n)\times (3n))| \\
	& &+& |P_{1,1}(T,(3n+3)\times (3n))| + |P_{1,1}(T,(3n+3)\times (3n))| \\
	& &+& |P_{1,1}(T,(3n+3)\times (3n))| \\
	&=& & A_{n} + A_{n} + A_{n}  \\
	& &+& A_{n} + A_{n} + A_{n}  \\
	& &+& C_{n} + C_{n} + C_{n}, 
\end{array}$}
\end{equation}
see also Figure~\ref{fig:recC3n2}.

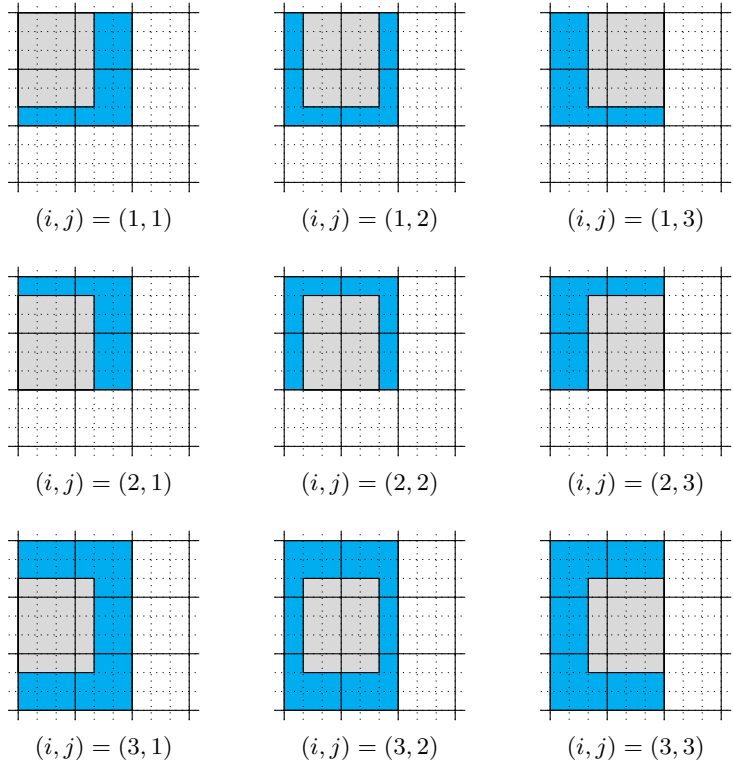
\begin{figure}[ht]
	\begin{center}
		\begin{tikzpicture}[scale = 0.25]
			\begin{scope}[xshift = 0cm, yshift = 28cm]			
				\draw[color=black, fill=gray, fill opacity=0.3] (0, 4) rectangle ++ (4, 5);
				\draw[color=black, fill=cyan, even odd rule] (0, 4) rectangle ++ (4, 5) (0, 3) rectangle ++ (6, 6);
				\draw[step=1, black, dotted]    (-0.5,-0.5) grid (9.5, 9.5);
				\draw[step=3, black, cap=rect]  (-0.5,-0.5) grid (9.5, 9.5);
				\draw(4.5, -2) node{\footnotesize $(i,j) = (1,1)$};
			\end{scope}
			\begin{scope}[xshift = 14cm, yshift = 28cm]
				\draw[color=black, fill=gray, fill opacity=0.3] (1, 4) rectangle ++ (4, 5);
				\draw[color=black, fill=cyan, even odd rule] (1, 4) rectangle ++ (4, 5) (0, 3) rectangle ++ (6, 6);
				\draw[step=1, black, dotted]    (-0.5,-0.5) grid (9.5, 9.5);
				\draw[step=3, black, cap=rect]  (-0.5,-0.5) grid (9.5, 9.5);
				\draw(4.5, -2) node{\footnotesize $(i,j) = (1,2)$};
			\end{scope}
			\begin{scope}[xshift = 28cm, yshift = 28cm]
				\draw[color=black, fill=gray, fill opacity=0.3] (2, 4) rectangle ++ (4, 5);
				\draw[color=black, fill=cyan, even odd rule] (2, 4) rectangle ++ (4, 5) (0, 3) rectangle ++ (6, 6);
				\draw[step=1, black, dotted]    (-0.5,-0.5) grid (9.5, 9.5);
				\draw[step=3, black, cap=rect]  (-0.5,-0.5) grid (9.5, 9.5);
				\draw(4.5, -2) node{\footnotesize $(i,j) = (1,3)$};
			\end{scope}
			\begin{scope}[xshift = 0cm, yshift = 14cm]			
				\draw[color=black, fill=gray, fill opacity=0.3] (0, 3) rectangle ++ (4, 5);
				\draw[color=black, fill=cyan, even odd rule] (0, 3) rectangle ++ (4, 5) (0, 3) rectangle ++ (6, 6);
				\draw[step=1, black, dotted]    (-0.5,-0.5) grid (9.5, 9.5);
				\draw[step=3, black, cap=rect]  (-0.5,-0.5) grid (9.5, 9.5);
				\draw(4.5, -2) node{\footnotesize $(i,j) = (2,1)$};
			\end{scope}
			\begin{scope}[xshift = 14cm, yshift = 14cm]
				\draw[color=black, fill=gray, fill opacity=0.3] (1, 3) rectangle ++ (4, 5);
				\draw[color=black, fill=cyan, even odd rule] (1, 3) rectangle ++ (4, 5) (0, 3) rectangle ++ (6, 6);
				\draw[step=1, black, dotted]    (-0.5,-0.5) grid (9.5, 9.5);
				\draw[step=3, black, cap=rect]  (-0.5,-0.5) grid (9.5, 9.5);
				\draw(4.5, -2) node{\footnotesize $(i,j) = (2,2)$};
			\end{scope}
			\begin{scope}[xshift = 28cm, yshift = 14cm]
				\draw[color=black, fill=gray, fill opacity=0.3] (2, 3) rectangle ++ (4, 5);
				\draw[color=black, fill=cyan, even odd rule] (2, 3) rectangle ++ (4, 5) (0, 3) rectangle ++ (6, 6);
				\draw[step=1, black, dotted]    (-0.5,-0.5) grid (9.5, 9.5);
				\draw[step=3, black, cap=rect]  (-0.5,-0.5) grid (9.5, 9.5);
				\draw(4.5, -2) node{\footnotesize $(i,j) = (2,3)$};
			\end{scope}
			\begin{scope}[xshift = 0cm, yshift = 0cm]			
				\draw[color=black, fill=gray, fill opacity=0.3] (0, 2) rectangle ++ (4, 5);
				\draw[color=black, fill=cyan, even odd rule] (0, 2) rectangle ++ (4, 5) (0, 0) rectangle ++ (6, 9);
				\draw[step=1, black, dotted]    (-0.5,-0.5) grid (9.5, 9.5);
				\draw[step=3, black, cap=rect]  (-0.5,-0.5) grid (9.5, 9.5);
				\draw(4.5, -2) node{\footnotesize $(i,j) = (3,1)$};
			\end{scope}
			\begin{scope}[xshift = 14cm, yshift = 0cm]
				\draw[color=black, fill=gray, fill opacity=0.3] (1, 2) rectangle ++ (4, 5);
				\draw[color=black, fill=cyan, even odd rule] (1, 2) rectangle ++ (4, 5) (0, 0) rectangle ++ (6, 9);
				\draw[step=1, black, dotted]    (-0.5,-0.5) grid (9.5, 9.5);
				\draw[step=3, black, cap=rect]  (-0.5,-0.5) grid (9.5, 9.5);
				\draw(4.5, -2) node{\footnotesize $(i,j) = (3,2)$};
			\end{scope}
			\begin{scope}[xshift = 28cm, yshift = 0cm]
				\draw[color=black, fill=gray, fill opacity=0.3] (2, 2) rectangle ++ (4, 5);
				\draw[color=black, fill=cyan, even odd rule] (2, 2) rectangle ++ (4, 5) (0, 0) rectangle ++ (6, 9);
				\draw[step=1, black, dotted]    (-0.5,-0.5) grid (9.5, 9.5);
				\draw[step=3, black, cap=rect]  (-0.5,-0.5) grid (9.5, 9.5);
				\draw(4.5, -2) node{\footnotesize $(i,j) = (3,3)$};
			\end{scope}
		\end{tikzpicture}
		\caption{Extensions used in \eqref{eq:recC3n2} for the recursion of $C_{3n-2}$. The gray region represent an element in $P_{i,j}(T,(3n-1)\times(3n-2))$, and the blue one the region it is extend with.}
     \label{fig:recC3n2}
	\end{center}
\end{figure}

\begin{equation}
\label{eq:recC3n1}
\renewcommand{\arraystretch}{1.5}
\text{\small$\begin{array}{c@{\ }c@{\ }c@{\ }l}
	C_{3n-1} 
	&=& \multicolumn{2}{l}{ \displaystyle{\sum_{i=1}^3 \sum_{j=1}^3} |P_{i,j}(T,(3n)\times (3n-1))| }\\
	&=& & |P_{1,1}(T,(3n)   \times (3n))|   + |P_{1,1}(T,(3n)  \times (3n))|  \\ 
	& &+& |P_{1,1}(T,(3n)   \times (3n+3))| + |P_{1,1}(T,(3n+3)\times (3n))|   \\ 
	& &+& |P_{1,1}(T,(3n+3) \times (3n))|   + |P_{1,1}(T,(3n+3)\times (3n+3))| \\
	& &+& |P_{1,1}(T,(3n+3) \times (3n))|   + |P_{1,1}(T,(3n+3)\times (3n))| \\
	& &+& |P_{1,1}(T,(3n+3) \times (3n+3))| \\
	&=& & A_{n} + A_{n} + B_{n}  \\
	& &+& C_{n} + C_{n} + A_{n+1}  \\
	& &+& C_{n} + C_{n} + A_{n+1},
\end{array}$}
\end{equation}
see also Figure~\ref{fig:recC3n1}.

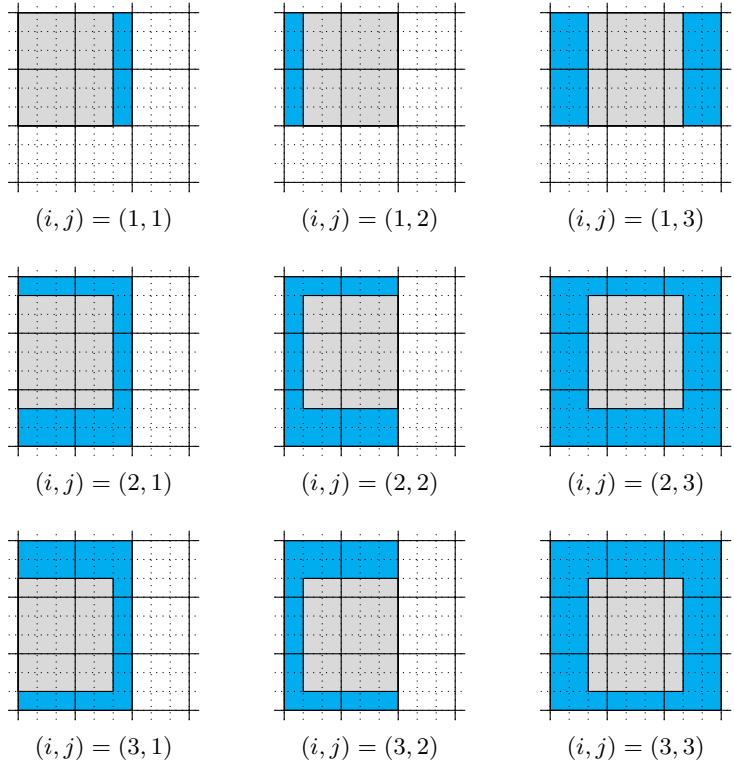
\begin{figure}[ht]
	\begin{center}
		\begin{tikzpicture}[scale = 0.25]
			\begin{scope}[xshift = 0cm, yshift = 28cm]			
				\draw[color=black, fill=gray, fill opacity=0.3] (0, 3) rectangle ++ (5, 6);
				\draw[color=black, fill=cyan, even odd rule] (0, 3) rectangle ++ (5, 6) (0, 3) rectangle ++ (6, 6);
				\draw[step=1, black, dotted]    (-0.5,-0.5) grid (9.5, 9.5);
				\draw[step=3, black, cap=rect]  (-0.5,-0.5) grid (9.5, 9.5);
				\draw(4.5, -2) node{\footnotesize $(i,j) = (1,1)$};
			\end{scope}
			\begin{scope}[xshift = 14cm, yshift = 28cm]
				\draw[color=black, fill=gray, fill opacity=0.3] (1, 3) rectangle ++ (5, 6);
				\draw[color=black, fill=cyan, even odd rule] (1, 3) rectangle ++ (5, 6) (0, 3) rectangle ++ (6, 6);
				\draw[step=1, black, dotted]    (-0.5,-0.5) grid (9.5, 9.5);
				\draw[step=3, black, cap=rect]  (-0.5,-0.5) grid (9.5, 9.5);
				\draw(4.5, -2) node{\footnotesize $(i,j) = (1,2)$};
			\end{scope}
			\begin{scope}[xshift = 28cm, yshift = 28cm]
				\draw[color=black, fill=gray, fill opacity=0.3] (2, 3) rectangle ++ (5, 6);
				\draw[color=black, fill=cyan, even odd rule] (2, 3) rectangle ++ (5, 6) (0, 3) rectangle ++ (9, 6);
				\draw[step=1, black, dotted]    (-0.5,-0.5) grid (9.5, 9.5);
				\draw[step=3, black, cap=rect]  (-0.5,-0.5) grid (9.5, 9.5);
				\draw(4.5, -2) node{\footnotesize $(i,j) = (1,3)$};
			\end{scope}
			\begin{scope}[xshift = 0cm, yshift = 14cm]			
				\draw[color=black, fill=gray, fill opacity=0.3] (0, 2) rectangle ++ (5, 6);
				\draw[color=black, fill=cyan, even odd rule] (0, 2) rectangle ++ (5, 6) (0, 0) rectangle ++ (6, 9);
				\draw[step=1, black, dotted]    (-0.5,-0.5) grid (9.5, 9.5);
				\draw[step=3, black, cap=rect]  (-0.5,-0.5) grid (9.5, 9.5);
				\draw(4.5, -2) node{\footnotesize $(i,j) = (2,1)$};
			\end{scope}
			\begin{scope}[xshift = 14cm, yshift = 14cm]
				\draw[color=black, fill=gray, fill opacity=0.3] (1, 2) rectangle ++ (5, 6);
				\draw[color=black, fill=cyan, even odd rule] (1, 2) rectangle ++ (5, 6) (0, 0) rectangle ++ (6, 9);
				\draw[step=1, black, dotted]    (-0.5,-0.5) grid (9.5, 9.5);
				\draw[step=3, black, cap=rect]  (-0.5,-0.5) grid (9.5, 9.5);
				\draw(4.5, -2) node{\footnotesize $(i,j) = (2,2)$};
			\end{scope}
			\begin{scope}[xshift = 28cm, yshift = 14cm]
				\draw[color=black, fill=gray, fill opacity=0.3] (2, 2) rectangle ++ (5, 6);
				\draw[color=black, fill=cyan, even odd rule] (2, 2) rectangle ++ (5, 6) (0, 0) rectangle ++ (9, 9);
				\draw[step=1, black, dotted]    (-0.5,-0.5) grid (9.5, 9.5);
				\draw[step=3, black, cap=rect]  (-0.5,-0.5) grid (9.5, 9.5);
				\draw(4.5, -2) node{\footnotesize $(i,j) = (2,3)$};
			\end{scope}
			\begin{scope}[xshift = 0cm, yshift = 0cm]			
				\draw[color=black, fill=gray, fill opacity=0.3] (0, 1) rectangle ++ (5, 6);
				\draw[color=black, fill=cyan, even odd rule] (0, 1) rectangle ++ (5, 6) (0, 0) rectangle ++ (6, 9);
				\draw[step=1, black, dotted]    (-0.5,-0.5) grid (9.5, 9.5);
				\draw[step=3, black, cap=rect]  (-0.5,-0.5) grid (9.5, 9.5);
				\draw(4.5, -2) node{\footnotesize $(i,j) = (3,1)$};
			\end{scope}
			\begin{scope}[xshift = 14cm, yshift = 0cm]
				\draw[color=black, fill=gray, fill opacity=0.3] (1, 1) rectangle ++ (5, 6);
				\draw[color=black, fill=cyan, even odd rule] (1, 1) rectangle ++ (5, 6) (0, 0) rectangle ++ (6, 9);
				\draw[step=1, black, dotted]    (-0.5,-0.5) grid (9.5, 9.5);
				\draw[step=3, black, cap=rect]  (-0.5,-0.5) grid (9.5, 9.5);
				\draw(4.5, -2) node{\footnotesize $(i,j) = (3,2)$};
			\end{scope}
			\begin{scope}[xshift = 28cm, yshift = 0cm]
				\draw[color=black, fill=gray, fill opacity=0.3] (2, 1) rectangle ++ (5, 6);
				\draw[color=black, fill=cyan, even odd rule] (2, 1) rectangle ++ (5, 6) (0, 0) rectangle ++ (9, 9);
				\draw[step=1, black, dotted]    (-0.5,-0.5) grid (9.5, 9.5);
				\draw[step=3, black, cap=rect]  (-0.5,-0.5) grid (9.5, 9.5);
				\draw(4.5, -2) node{\footnotesize $(i,j) = (3,3)$};
			\end{scope}
		\end{tikzpicture}
		\caption{Extensions used in \eqref{eq:recC3n1} for the recursion of $C_{3n-1}$. The gray region represent an element in $P_{i,j}(T,(3n)\times(3n-1))$, and the blue one the region it is extend with.}
     \label{fig:recC3n1}
	\end{center}
\end{figure}

\begin{equation}
\label{eq:recC3n}
\renewcommand{\arraystretch}{1.5}
\text{\small$\begin{array}{c@{\ }c@{\ }c@{\ }l}
	C_{3n} 
	&=& \multicolumn{2}{l}{ \displaystyle{\sum_{i=1}^3 \sum_{j=1}^3} |P_{i,j}(T,(3n+1)\times (3n))| }\\
	&=& & |P_{1,1}(T,(3n+3)\times (3n))|   + |P_{1,1}(T,(3n+3)\times (3n+3))| \\ 
	& &+& |P_{1,1}(T,(3n+3)\times (3n+3))| + |P_{1,1}(T,(3n+3)\times (3n))| \\ 
	& &+& |P_{1,1}(T,(3n+3)\times (3n+3))| + |P_{1,1}(T,(3n+3)\times (3n+3))| \\
	& &+& |P_{1,1}(T,(3n+3)\times (3n))|   + |P_{1,1}(T,(3n+3)\times (3n+3))| \\
	& &+& |P_{1,1}(T,(3n+3)\times (3n+3))| \\
	&=& & C_{n} + A_{n+1} + A_{n+1}  \\
	& &+& C_{n} + A_{n+1} + A_{n+1}  \\
	& &+& C_{n} + A_{n+1} + A_{n+1}, 
\end{array}$}
\end{equation}
see also Figure~\ref{fig:recC3n}.

\begin{figure}[ht]
	\begin{center}
		\begin{tikzpicture}[scale = 0.25]
			\begin{scope}[xshift = 0cm, yshift = 28cm]			
				\draw[color=black, fill=gray, fill opacity=0.3] (0, 2) rectangle ++ (6, 7);
				\draw[color=black, fill=cyan, even odd rule] (0, 2) rectangle ++ (6, 7) (0, 0) rectangle ++ (6, 9);
				\draw[step=1, black, dotted]    (-0.5,-0.5) grid (9.5, 9.5);
				\draw[step=3, black, cap=rect]  (-0.5,-0.5) grid (9.5, 9.5);
				\draw(4.5, -2) node{\footnotesize $(i,j) = (1,1)$};
			\end{scope}
			\begin{scope}[xshift = 14cm, yshift = 28cm]
				\draw[color=black, fill=gray, fill opacity=0.3] (1, 2) rectangle ++ (6, 7);
				\draw[color=black, fill=cyan, even odd rule] (1, 2) rectangle ++ (6, 7) (0, 0) rectangle ++ (9, 9);
				\draw[step=1, black, dotted]    (-0.5,-0.5) grid (9.5, 9.5);
				\draw[step=3, black, cap=rect]  (-0.5,-0.5) grid (9.5, 9.5);
				\draw(4.5, -2) node{\footnotesize $(i,j) = (1,2)$};
			\end{scope}
			\begin{scope}[xshift = 28cm, yshift = 28cm]
				\draw[color=black, fill=gray, fill opacity=0.3] (2, 2) rectangle ++ (6, 7);
				\draw[color=black, fill=cyan, even odd rule] (2, 2) rectangle ++ (6, 7) (0, 0) rectangle ++ (9, 9);
				\draw[step=1, black, dotted]    (-0.5,-0.5) grid (9.5, 9.5);
				\draw[step=3, black, cap=rect]  (-0.5,-0.5) grid (9.5, 9.5);
				\draw(4.5, -2) node{\footnotesize $(i,j) = (1,3)$};
			\end{scope}
			\begin{scope}[xshift = 0cm, yshift = 14cm]			
				\draw[color=black, fill=gray, fill opacity=0.3] (0, 1) rectangle ++ (6, 7);
				\draw[color=black, fill=cyan, even odd rule] (0, 1) rectangle ++ (6, 7) (0, 0) rectangle ++ (6, 9);
				\draw[step=1, black, dotted]    (-0.5,-0.5) grid (9.5, 9.5);
				\draw[step=3, black, cap=rect]  (-0.5,-0.5) grid (9.5, 9.5);
				\draw(4.5, -2) node{\footnotesize $(i,j) = (2,1)$};
			\end{scope}
			\begin{scope}[xshift = 14cm, yshift = 14cm]
				\draw[color=black, fill=gray, fill opacity=0.3] (1, 1) rectangle ++ (6, 7);
				\draw[color=black, fill=cyan, even odd rule] (1, 1) rectangle ++ (6, 7) (0, 0) rectangle ++ (9, 9);
				\draw[step=1, black, dotted]    (-0.5,-0.5) grid (9.5, 9.5);
				\draw[step=3, black, cap=rect]  (-0.5,-0.5) grid (9.5, 9.5);
				\draw(4.5, -2) node{\footnotesize $(i,j) = (2,2)$};
			\end{scope}
			\begin{scope}[xshift = 28cm, yshift = 14cm]
				\draw[color=black, fill=gray, fill opacity=0.3] (2, 1) rectangle ++ (6, 7);
				\draw[color=black, fill=cyan, even odd rule] (2, 1) rectangle ++ (6, 7) (0, 0) rectangle ++ (9, 9);
				\draw[step=1, black, dotted]    (-0.5,-0.5) grid (9.5, 9.5);
				\draw[step=3, black, cap=rect]  (-0.5,-0.5) grid (9.5, 9.5);
				\draw(4.5, -2) node{\footnotesize $(i,j) = (2,3)$};
			\end{scope}
			\begin{scope}[xshift = 0cm, yshift = 0cm]			
				\draw[color=black, fill=gray, fill opacity=0.3] (0, 0) rectangle ++ (6, 7);
				\draw[color=black, fill=cyan, even odd rule] (0, 0) rectangle ++ (6, 7) (0, 0) rectangle ++ (6, 9);
				\draw[step=1, black, dotted]    (-0.5,-0.5) grid (9.5, 9.5);
				\draw[step=3, black, cap=rect]  (-0.5,-0.5) grid (9.5, 9.5);
				\draw(4.5, -2) node{\footnotesize $(i,j) = (3,1)$};
			\end{scope}
			\begin{scope}[xshift = 14cm, yshift = 0cm]
				\draw[color=black, fill=gray, fill opacity=0.3] (1, 0) rectangle ++ (6, 7);
				\draw[color=black, fill=cyan, even odd rule] (1, 0) rectangle ++ (6, 7) (0, 0) rectangle ++ (9, 9);
				\draw[step=1, black, dotted]    (-0.5,-0.5) grid (9.5, 9.5);
				\draw[step=3, black, cap=rect]  (-0.5,-0.5) grid (9.5, 9.5);
				\draw(4.5, -2) node{\footnotesize $(i,j) = (3,2)$};
			\end{scope}
			\begin{scope}[xshift = 28cm, yshift = 0cm]
				\draw[color=black, fill=gray, fill opacity=0.3] (2, 0) rectangle ++ (6, 7);
				\draw[color=black, fill=cyan, even odd rule] (2, 0) rectangle ++ (6, 7) (0, 0) rectangle ++ (9, 9);
				\draw[step=1, black, dotted]    (-0.5,-0.5) grid (9.5, 9.5);
				\draw[step=3, black, cap=rect]  (-0.5,-0.5) grid (9.5, 9.5);
				\draw(4.5, -2) node{\footnotesize $(i,j) = (3,3)$};
			\end{scope}
		\end{tikzpicture}
		\caption{Extensions used in \eqref{eq:recC3n} for the recursion of $C_{3n}$. The gray region represent an element in $P_{i,j}(T,(3n+1)\times(3n))$, and the blue one the region it is extend with.}
     \label{fig:recC3n}
	\end{center}
\end{figure}
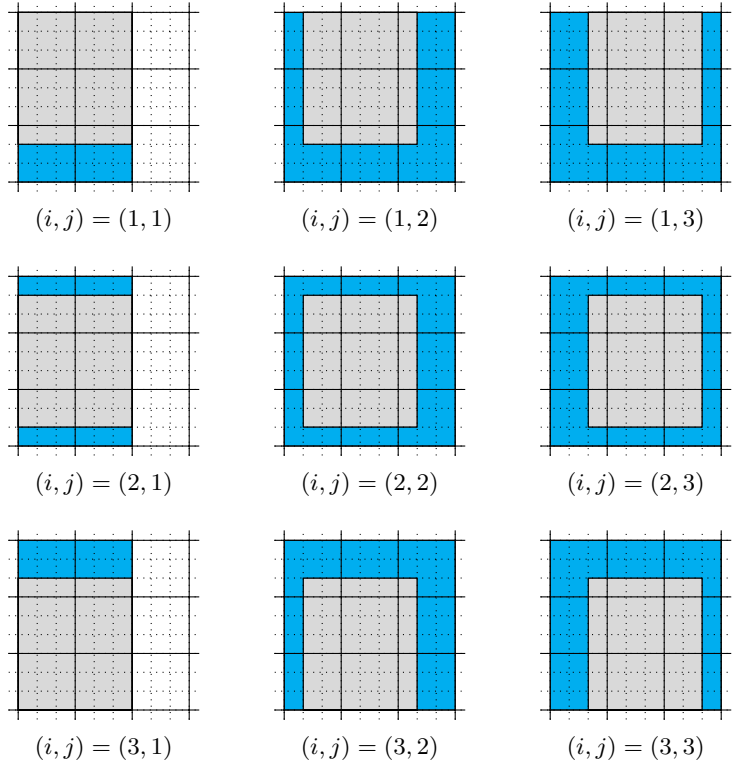

We have now gone through all the necessary cases for the recursions. The initial values for these recursion are obtained via a straight forward enumeration, see Table~\ref{table:initialvalues}. 

\begin{table}[ht]
\footnotesize
\centering
\begin{tabular}{c*{11}{r}}
\toprule
$n$ & 
\multicolumn{1}{p{7mm}}{\hfill 1} & 
\multicolumn{1}{p{7mm}}{\hfill 2} & 
\multicolumn{1}{p{7mm}}{\hfill 3} & 
\multicolumn{1}{p{7mm}}{\hfill 4} & 
\multicolumn{1}{p{7mm}}{\hfill 5} & 
\multicolumn{1}{p{7mm}}{\hfill 6} & 
\multicolumn{1}{p{7mm}}{\hfill 7} & 
\multicolumn{1}{p{7mm}}{\hfill 8} &
\multicolumn{1}{p{7mm}}{\hfill 9} & 
\multicolumn{1}{p{7mm}}{\hfill10} \\ 
\midrule
$A_n$      &    2 &   14 &   70 &  126 &  270 &  438 &  630 &  790 &  958 & 1134 \\
$B_n$      &    4 &   36 &   96 &  192 &  348 &  528 &  708 &  872 & 1044 & 1332 \\
$C_n$      &    4 &   36 &   96 &  192 &  348 &  528 &  708 &  872 & 1044 & 1332 \\
\bottomrule
\end{tabular}
\caption{Initial terms for $A$, $B$, and $C$. }.
\label{table:initialvalues}
\end{table}

\section{Proof of Main Theorem}
\label{sec:mainproof}

In the previous section we derived recursions for $A_n, B_n$, and $C_n$ and in Table~\ref{table:initialvalues} we presented their initial values. In this section we show how to solve this recursion system, and thereby prove Theorem~\ref{thm:main}.

\begin{lemma}
\label{lemma:BequalsC}
Let $n\geq1$. Then $B_n = C_n$.
\end{lemma}
\begin{proof}
Let us consider the following three cases,
\begin{equation}
\label{eq:BequalsC}
\renewcommand{\arraystretch}{1.3}
\left\{
	\begin{array}{r@{\ }l}
	B_{3k-2} &= C_{3k-2},\\
	B_{3k-1} &= C_{3k-1}, \\
	B_{3k}   &= C_{3k}. 
\end{array}
\right.
\end{equation}
We prove the equalities in \eqref{eq:BequalsC} by induction on $k$. The basis cases, $k=1,2$ are directly seen in Table~\ref{table:initialvalues}.
Assume for induction that the equalities in \eqref{eq:BequalsC} holds for $k < p$. Then the recursions for $B$ and $C$, and the induction assumption give 
\begin{align*}
	B_{3p-2} - C_{3p-2} &= 6A_{p} + 3B_{p} - 6A_{p} - 3C_{p}  = 3(B_{p} - C_{p}) =  0.
\end{align*}
Similar, for the second equality we have 
\begin{align*}
	B_{3p-1} - C_{3p-1} 
	&= 2A_{p} + 4B_{p} +  C_{p} + 2A_{p+1} \\
	&\quad   - 2A_{p} -  B_{p} - 4C_{p} - 2A_{p+1}  \\
	&= 3(B_{p} - C_{p}) \\
	&=  0.
\end{align*}
And in the same way, 
\begin{align*}
	B_{3p} - C_{3p} &= 3B_{p} + 6A_{p+1} - 3C_{p} - 6A_{p+1} = 3(B_{p} - C_{p}) =  0,
\end{align*}
which completes the induction. 
\end{proof}

From the recursion \eqref{eq:recA3n1}, \eqref{eq:recB3n}, \eqref{eq:recA3n2}, and combined with Lemma~\ref{lemma:BequalsC} we obtain for $n\geq 2$ 
\begin{align*}
	A_{9n-1}  
	& = 4A_{3n} + 4B_{3n} + A_{3n+1} + 3(A_{3n} - A_{n} - 4B_{n} - 4A_{n+1}) \\
	& = 7A_{3n} + 4(3B_{n} +6A_{n+1}) + A_{3n+1} - 3A_{n} - 12B_{n} - 12A_{n+1} \\
	& = 7A_{3n} + 12A_{n+1} + A_{3n+1} - 3A_{n}  \\
	& = 7A_{3n} + \frac{7}{3} A_{3n+1}  - \frac{1}{3}A_{3n-2}. 
\end{align*}
Continuing in the same way we obtain
\begin{equation}
\label{eq:rec}
\renewcommand{\arraystretch}{1.5}
	\left\{
	\begin{array}{l@{\ }c*{4}{@{\ }c@{\ }c@{}l}}
		A_{9n-8} &=& &           9 & A_{3n-2}, \\
		A_{9n-7} &=& &\frac{19}{3} & A_{3n-2} &+&   & A_{3n-1} &+&           3 & A_{3n}   &-& \frac{4}{3} & A_{3n+1}, \\
		A_{9n-6} &=& &\frac{10}{3} & A_{3n-2} &+& 4 & A_{3n-1} &+&           3 & A_{3n}   &-& \frac{4}{3} & A_{3n+1}, \\
		A_{9n-5} &=& &           9 & A_{3n-1}, \\
		A_{9n-4} &=& & \frac{1}{3} & A_{3n-2} &+& 4 & A_{3n-1} &+&           6 & A_{3n}   &-& \frac{4}{3} & A_{3n+1}, \\
		A_{9n-3} &=& & \frac{1}{3} & A_{3n-2} &+&   & A_{3n-1} &+&           9 & A_{3n}   &-& \frac{4}{3} & A_{3n+1}, \\
		A_{9n-2} &=& &           9 & A_{3n},  \\
		A_{9n-1} &=&-& \frac{1}{3} & A_{3n-2} &+& 7 & A_{3n}   &+&  \frac{7}{3} & A_{3n+1},   \\
		A_{9n}   &=&-& \frac{1}{3} & A_{3n-2} &+& 4 & A_{3n}   &+& \frac{16}{3} & A_{3n+1},  
	\end{array}
\right.
\end{equation}
for $n\geq 2$. The above recursions can be simplified a little.
\begin{lemma}
The number of square patterns in the squiral tiling $T$ fulfil the recursions 
\begin{equation}
\label{eq:simprec}
	\renewcommand{\arraystretch}{1.5}
	\left\{
	\begin{array}{l@{\ }*{3}{@{\ }c@{\ }r@{}l}}
		A_{3n-2} &=&  9 & A_{n}, \\
		A_{9n-7} &=&  5 & A_{3n+1} &-& 16 & A_{3n}  &+& 20 & A_{3n-1}, \\
		A_{9n-4} &=&  - & A_{3n+1} &+&  5 & A_{3n}  &+& 5 & A_{3n-1},  \\
		A_{9n-1} &=&  2 & A_{3n+1} &+&  8 & A_{3n}  &-&   & A_{3n-1},  \\
		A_{3n}   &=&    & A_{3n-1} &+&  3 & A_{n+1} &-& 3 & A_{n},  
	\end{array}
	\right.
\end{equation}
with the initial values of $A_i$, for $i = 1,\ldots,8$, given in Table \ref{table:initialvalues}.
\end{lemma}
\begin{proof}
The first case in \eqref{eq:simprec} is direct from \eqref{eq:rec}. 
By looking at the three differences; $A_{9n-i}-A_{9n-i-1}$, where $i=\{0,3,6\}$, we may conclude  
\[	
	A_{3n} =  A_{3n-1}  + 3 (A_{n+1} - A_{n}),
\]
which is the final case in \eqref{eq:simprec}. The last three cases are from \eqref{eq:rec} by adding the just above derived equality as follows 
\begin{align*}
	A_{9n-1} &= -\frac{1}{3} A_{3n-2} + 7A_{3n} + \frac{7}{3}A_{3n+1} + \big(A_{3n} - A_{3n-1}  - 3 (A_{n+1} - A_{n})\big)\\
			 &= -\frac{1}{3} A_{3n-2} + 7A_{3n} + \frac{7}{3}A_{3n+1} +  A_{3n} - A_{3n-1}  - \frac{1}{3} A_{3n+1} + \frac{1}{3}A_{3n-2}   \\
			 &=  7A_{3n} + \frac{7}{3}A_{3n+1} +  A_{3n} - A_{3n-1}  - \frac{1}{3} A_{3n+1}   \\
			 &=  - A_{3n-1} + 8A_{3n} + 2A_{3n+1}.  
\end{align*}
The remaining cases follow by looking at $A_{9n-4} - \big(A_{3n} - A_{3n-1}  - 3 (A_{n+1} - A_{n})\big)$ and 
$A_{9n-7} - 19\big(A_{3n} - A_{3n-1}  - 3 (A_{n+1} - A_{n})\big)$.
\end{proof}

For $n\geq4$ define the two integer valued functions  
\[
    \alpha(n) := \big\lfloor \log_3 (n-2) \big\rfloor,
    \quad\textnormal{and}\quad
    \beta(n) := \left\lfloor \log_3 \frac{n-2}{2} \right\rfloor,
\]
where $\log_3$ denotes the logarithm in base 3 and the brackets $\lfloor \cdot \rfloor$ is the floor function; see also the statement of Theorem~\ref{thm:main} for $\alpha$ and $\beta$. 

To prove that \eqref{eq:main} fulfils the first recursion relation in \eqref{eq:simprec}, $(A_{3n-2} = 9A_n)$, let us denote $\alpha_0 := \alpha(3n-2)$ and $\alpha_1 := \alpha(n)$, and similarly for $\beta_0$ and $\beta_1$. Then we see in Table~\ref{table:alphabetavalues1} that $\alpha_0 = \alpha_1+1$ and $\beta_0 = \beta_1+1$. This leads to 
\[
\text{\footnotesize\renewcommand{\arraystretch}{1.5}%
$\begin{array}{c@{\ }c@{\ }c@{\ }c@{}l@{\ }c@{\ }l@{\ }c@{\ }l}
	A_{3n-2} &-& \multicolumn{7}{@{\ }l}{ 9 A_n}\\
	         &=& &        & \big(4+8\alpha_0 -8\beta_0\big)(3n-3)^2 &+&  \big(12\cdot3^{\alpha_0}+24\cdot3^{\beta_0}\big)(3n-3) &-& 18 \cdot 9^{\alpha_0} \\
	         & &-& 9 \Big(& \big(4+8\alpha_1 -8\beta_1\big)(n-1)^2  &+&  \big(12\cdot3^{\alpha_1}+24\cdot3^{\beta_1}\big) (n-1) &-& 18 \cdot 9^{\alpha_1}\Big) \\
			 &=& \multicolumn{7}{@{\ }l}{0.}
\end{array}$}
\]
\begin{table}[ht]
\centering
{\small
\renewcommand{\arraystretch}{1.1}
\begin{tabular}{l@{\hspace{7mm}}cccccc}
\toprule
      \multicolumn{1}{c}{$n$} 
    & \multicolumn{1}{p{15mm}}{\centering\arraybackslash$3^k+1$} 
    & \multicolumn{1}{p{15mm}}{\centering\arraybackslash$3^k+2$} 
    & \multicolumn{1}{p{15mm}}{\centering\arraybackslash$2\cdot3^k+1$} 
    & \multicolumn{1}{p{15mm}}{\centering\arraybackslash$2\cdot3^k+2$} 
\\
\midrule
$\alpha(3n-2)$  & $k$   & $k+1$  & $k+1$ & $k+1$ \\ 
$\beta( 3n-2)$  & $k$   & $k$    & $k$   & $k+1$ \\
\midrule
$\alpha(n)$     & $k-1$ & $k$    & $k$   & $k$   \\ 
$\beta( n)$     & $k-1$ & $k-1$  & $k-1$ & $k$   \\
\bottomrule
\end{tabular}
}
\caption{Values of $\alpha$ and $\beta$. }
\label{table:alphabetavalues1}
\end{table}

The second, third and fourth recursion of \eqref{eq:simprec} follow all the same scheme. Therefore let us only consider the second recursion. Similar to the case above we use here the short hand; $\alpha_0  := \alpha(9n-7)$, $\alpha_1  := \alpha(3n+1)$, $\alpha_2  := \alpha(3n)$, and $\alpha_3  := \alpha(3n-1)$. Analogous we use the $\beta_0, \ldots, \beta_3$. From Table~\ref{table:alphabetavalues2} we see the different values of the $\alpha$ and $\beta$ depending on $n$. Here we have to consider 
the two intervals for $n$, namely
\[	
	3^k +1 \leq n < 2 \cdot 3^k + 1, \quad \textnormal{and}\quad 2 \cdot 3^k + 1 \leq n <  3^{k+1} + 1.
\]
In both these cases we find 
\[
\text{\footnotesize\renewcommand{\arraystretch}{1.6}%
$\begin{array}{c@{\ }c@{\ }c@{\ }r@{}l@{\ }c@{\ }l@{\ }c@{\ }l}
	A_{9n-7}   &-& \multicolumn{7}{@{\,}l}{ 5A_{3n+1} +16 A_{3n} -20 A_{3n-1}}\\
	         &=& &         & \big(4+8\alpha_0 -8\beta_0\big)(9n-8)^2   &+&  \big(12\cdot3^{\alpha_0}+24\cdot3^{\beta_0}\big)(9n-8)  &-& 18 \cdot 9^{\alpha_0} \\
	         & &-&  5 \Big(& \big(4+8\alpha_1 -8\beta_1\big)(3n  )^2   &+&  \big(12\cdot3^{\alpha_1}+24\cdot3^{\beta_1}\big)(3n  )  &-& 18 \cdot 9^{\alpha_1}\Big) \\
	         & &+& 16 \Big(& \big(4+8\alpha_2 -8\beta_2\big)(3n-1)^2   &+&  \big(12\cdot3^{\alpha_2}+24\cdot3^{\beta_2}\big)(3n-1)  &-& 18 \cdot 9^{\alpha_2}\Big) \\
	         & &-& 20 \Big(& \big(4+8\alpha_3 -8\beta_3\big)(3n-2)^2   &+&  \big(12\cdot3^{\alpha_3}+24\cdot3^{\beta_3}\big)(3n-2)  &-& 18 \cdot 9^{\alpha_3}\Big) \\
			 &=& \multicolumn{7}{@{\ }l}{0.}
\end{array}$}
\]
The recursion for $A_{9n-4}$, and $A_{9n-1}$ are treated in the same way. 

\begin{table}[ht]
\centering
{\small\renewcommand{\arraystretch}{1.1}
\begin{tabular}{l@{\hspace{7mm}}cccccc}
\toprule
      \multicolumn{1}{c}{$n$} 
    & \multicolumn{1}{p{12mm}}{\centering\arraybackslash$3^k$}
    & \multicolumn{1}{p{12mm}}{\centering\arraybackslash$3^k+1$}
    & \multicolumn{1}{p{12mm}}{\centering\arraybackslash$3^k+2$}
    & \multicolumn{1}{p{12mm}}{\centering\arraybackslash$2\cdot3^k$} 
    & \multicolumn{1}{p{12mm}}{\centering\arraybackslash$2\cdot3^k+1$} 
    & \multicolumn{1}{p{12mm}}{\centering\arraybackslash$2\cdot3^k+2$} 
 \\
\midrule
$\alpha(9n-7)$  & $k+1$ & $k+2$ & $k+2$ & $k+2$ & $k+2$ & $k+2$ \\ 
$\beta( 9n-7)$  & $k+1$ & $k+1$ & $k+1$ & $k+1$ & $k+2$ & $k+2$ \\ 
\midrule
$\alpha(9n-4)$  & $k+1$ & $k+2$ & $k+2$ & $k+2$ & $k+2$ & $k+2$ \\ 
$\beta( 9n-4)$  & $k+1$ & $k+1$ & $k+1$ & $k+1$ & $k+2$ & $k+2$ \\ 
\midrule
$\alpha(9n-1)$  & $k+1$ & $k+2$ & $k+2$ & $k+2$ & $k+2$ & $k+2$ \\ 
$\beta( 9n-1)$  & $k+1$ & $k+1$ & $k+1$ & $k+1$ & $k+2$ & $k+2$ \\ 
\midrule
$\alpha(3n+1)$  & $k$   & $k+1$ & $k+1$ & $k+1$ & $k+1$ & $k+1$ \\ 
$\beta( 3n+1)$  & $k$   & $k$   & $k$   & $k$   & $k+1$ & $k+1$ \\ 
\midrule
$\alpha(3n)$    & $k$   & $k+1$ & $k+1$ & $k+1$ & $k+1$ & $k+1$ \\ 
$\beta( 3n)$    & $k$   & $k$   & $k$   & $k$   & $k+1$ & $k+1$ \\
\midrule
$\alpha(3n-1)$  & $k$   & $k+1$ & $k+1$ & $k+1$ & $k+1$ & $k+1$ \\ 
$\beta( 3n-1)$  & $k$   & $k$   & $k$   & $k$   & $k+1$ & $k+1$ \\
\bottomrule
\end{tabular}
}
\caption{Values of $\alpha$ and $\beta$. }
\label{table:alphabetavalues2}
\end{table}

For the final term in \eqref{eq:simprec} we proceed in the same way. From Table~\ref{table:alphabetavalues3} we see that we have to consider 4 different cases for $n$, namely; 
$n = 3^k +1$, $n = 2\cdot 3^k + 1$,
\[	
	3^k+2\leq n < 2 \cdot 3^k + 1, \quad \textnormal{and}\quad 2\cdot 3^k + 2  \leq n < 3^{k+1} + 1.
\]
As above, we find for all the cases $A_{3n} - A_{3n-1} - 3 A_{n+1} + 3 A_{n} = 0$. Therefore we may conclude that \eqref{eq:main} is a solution to \eqref{eq:simprec}, which completes the proof of Theorem~\ref{thm:main}. 

\begin{table}[ht]
\centering
{\small
\renewcommand{\arraystretch}{1.1}
\begin{tabular}{l@{\hspace{7mm}}cccccc}
\toprule
      \multicolumn{1}{c}{$n$} 
    & \multicolumn{1}{p{12mm}}{\centering\arraybackslash$3^k$} 
    & \multicolumn{1}{p{12mm}}{\centering\arraybackslash$3^k+1$} 
    & \multicolumn{1}{p{12mm}}{\centering\arraybackslash$3^k+2$} 
    & \multicolumn{1}{p{12mm}}{\centering\arraybackslash$2\cdot3^k$} 
    & \multicolumn{1}{p{12mm}}{\centering\arraybackslash$2\cdot3^k+1$} 
    & \multicolumn{1}{p{12mm}}{\centering\arraybackslash$2\cdot3^k+2$} 
\\
\midrule
$\alpha(3n)$    &$k$   & $k+1$ & $k+1$  & $k+1$  & $k+1$ & $k+1$ \\ 
$\beta( 3n)$    &$k$   & $k$   & $k$    & $k$    & $k+1$ & $k+1$ \\
\midrule
$\alpha(3n-1)$  &$k$   & $k+1$ & $k+1$  & $k+1$  & $k+1$ & $k+1$ \\ 
$\beta( 3n-1)$  &$k$   & $k$   & $k$    & $k$    & $k+1$ & $k+1$ \\
\midrule
$\alpha(n+1)$   &$k-1$ & $k$   & $k$    & $k$    & $k$   & $k$   \\ 
$\beta( n+1)$   &$k-1$ & $k-1$ & $k-1$  & $k-1$  & $k$   & $k$   \\
\midrule
$\alpha(n)$     &$k-1$ & $k-1$ & $k$    & $k$    & $k$   & $k$   \\ 
$\beta( n)$     &$k-1$ & $k-1$ & $k-1$  & $k-1$  & $k-1$ & $k$   \\
\bottomrule
\end{tabular}
}
\caption{Values of $\alpha$ and $\beta$.}
\label{table:alphabetavalues3}
\end{table}

Let us end with a short note of precaution when calculating the values of $A_n$ via \eqref{eq:main}. The standard implementation for numerical calculations using double precision (the IEEE 754), e.g.\ as used in the programming language JAVA, gives the approximated value 
\[
	5 = \frac{\ln 243 }{\ln 3}  \approx 4.999999999999999,  
\]
which leads to the wrong value of $\alpha(243+2)$.

\noindent\rule{10em}{0.5pt}

\texttt{\small johannilsson514@gmail.com}
\end{document}